\documentclass[a4paper, 11pt]{article}

\usepackage[T1]{fontenc}
\usepackage[utf8]{inputenc}

\usepackage{amssymb}
\usepackage{amsmath}
\usepackage{amsfonts}
\usepackage{amsthm}
\usepackage{mathtools}
 
\usepackage{comment}
\usepackage{url}
\usepackage{float}
\usepackage[dvipsnames]{xcolor}
\usepackage{enumerate}
\usepackage{tikz}
\usepackage{caption}
\usepackage{cases}

\newcommand{\E}{\mathbb{E}}
\newcommand{\eps}{\epsilon}

\theoremstyle{plain}%
\newtheorem{theorem}{Theorem}[section]
\newtheorem{lemma}[theorem]{Lemma}
\newtheorem{proposition}[theorem]{Proposition}
\newtheorem{corollary}[theorem]{Corollary}

\newtheorem{question}[theorem]{Question}

\theoremstyle{definition}
\newtheorem{definition}[theorem]{Definition}
\newtheorem{example}[theorem]{Example}

\theoremstyle{remark}
\newtheorem{remark}[theorem]{Remark}

\def\calF{{\mathcal F}}

\newcommand{\g}[1]
{g_{_{#1}}}

\newcommand{\rc}[1]
{\textcolor{red}{#1}}

\newcommand{\br}[1]
{\left(#1\right)}

\usepackage{url}
\usepackage{hyperref}%
\hypersetup{
	colorlinks=false,
	linkbordercolor=gray,
	citebordercolor=gray,
	urlbordercolor=gray}%

\title{Where to stand when playing darts?}

\author{Björn G. Franzén
\thanks{Chalmers University of Technology, Gothenburg, 
Sweden.\ \ Email:
        \hbox{bjorn.gunnar.franzen@gmail.com}}
\and         Jeffrey E. Steif
\thanks{Chalmers University of Technology and Gothenburg University, Gothenburg, Sweden.\ \ Email:
        \hbox{steif@chalmers.se}}
\and         Johan Wästlund
\thanks{Chalmers University of Technology and Gothenburg University, Gothenburg, Sweden.\ \ Email:
        \hbox{wastlund@chalmers.se}}
}

\date{\today}

\begin{document}

\maketitle

\begin{abstract}

This paper analyzes the question of where one should stand when playing darts.
If one stands at distance $d>0$ and aims at $a\in \mathbb{R}^n$, then the dart (modelled by a random vector $X$ in $\mathbb{R}^n$)
hits a random point given by $a+dX$. Next, given a payoff function $f$, one considers
$$
\sup_a Ef(a+dX)
$$
and asks if this is decreasing in $d$;  i.e., whether it is better to stand closer rather than farther
from the target.  Perhaps surprisingly, this is not always the case and understanding 
when this does or does not occur is the purpose of this paper.

We show that if $X$ has a so-called {\em selfdecomposable} distribution, then it is always better
to stand closer for any payoff function. This class includes all stable distributions as well as many more.

On the other hand, if the payoff function is $\cos(x)$, then it is always better to stand closer
if and only if the characteristic function $|\phi_X(t)|$ is decreasing on $[0,\infty)$.
We will then show that if there are at least two point masses, then it is not always better to stand closer using $\cos(x)$.
If there is a single point mass, one can find a different payoff function to obtain this phenomenon.

Another large class of darts $X$ for which there are bounded continuous payoff functions for which
it is not always better to stand closer are distributions with compact
support. This will be obtained by using the fact that the Fourier transform of such distributions 
has a zero in the complex plane. This argument will work whenever there is a complex zero of the Fourier transform.

Finally, we analyze if the property of it being better to stand closer is
closed under convolution and/or limits.

 \newline
 MSC 2010 \emph{subject classifications.}
 Primary 60E10
  \medskip\noindent
\end{abstract}



\tableofcontents


\bigskip

\bigskip

\section{Introduction}

\subsection{Model and Main Results}

We begin immediately by formalizing the notion of a general dart game in $\mathbb{R}^n$.


\begin{definition} A \textbf{dart} is a random vector $X$ taking values in $\mathbb{R}^n$. 
It represents the distribution of where you hit the target ($\mathbb{R}^n$)
when you stand at distance one and aim at the origin.
\end{definition}

\begin{definition} 
A \textbf{payoff function} $f$ is a  measurable 
function from $\mathbb{R}^n$ to $\mathbb{R}$ which is bounded from above. 
\end{definition} 

Given a dart $X$ and a player standing at distance $d$ aiming at $a\in \mathbb{R}$, 
the distribution of where she hits the target is modelled by $a+dX$.
Assuming you want to maximize the expected payoff with respect to a given payoff function $f$
from a certain distance, it is natural to consider the function
\begin{definition}
\begin{equation}
    g_{_{X,f}}(d):=\sup_{a\in\mathbb{R}^n} Ef(a+dX)
\end{equation}
defined for $d>0$.
\end{definition}
So $g_{_{X,f}}(d)$ is the best you can achieve
with dart $X$, standing at distance $d$ with 
payoff function $f$. Note that the supremum is not always assumed.

\medskip\noindent
{\bf Question:} Is it always better to stand closer to the target? I.e., is $g_{_{X,f}}(d)$ a decreasing function of $d$?

\medskip
Perhaps surprisingly, the answer is no.
We start off by quickly giving a simple example showing
that this is not necessarily the case.
In one dimension, let $X$ be uniform on $[0,2]$ and $f$ be 1 on intervals of the form $[2k,2k+1]$
and 0 on intervals of the form $(2k-1,2k)$ where $k$ is an integer. It is then immediate to 
check that $\g{X,f}(1)=1/2$ (and it doesn't matter where you aim) but $\g{X,f}(3/2)=2/3$
(by aiming e.g.\ at 1.5). We will later see how this is related to a more general 
phenomenon where the behavior of the characteristic function of $X$ will play a central role, 
see Theorem \ref{theorem:cfeasy}.

We introduce the following concepts which capture those situations where standing closer is 
better.

\begin{definition} 
The pair $(X,f)$ is \textbf{reasonable} if $g_{_{X,f}}(d)$ is decreasing in $d$. 
The dart $X$ is \textbf{reasonable with respect to} a family $\calF$ of payoff functions 
if $(X,f)$ is reasonable for all $f\in \calF$. If $(X,f)$ is reasonable for all 
payoff functions $f$, then $X$ is said to be \textbf{reasonable}. The payoff function 
$f$ is \textbf{reasonable with respect to} a family $\mathcal{X}$  of darts
if $(X,f)$ is reasonable for all $X\in \mathcal{X}$. If $(X,f)$ is reasonable for all darts 
$X$, then $f$ is said to be \textbf{reasonable}.
\end{definition} 
We will often use the expression that ``$X$ is reasonable with respect to $f$'' to mean 
that $(X,f)$ is reasonable.
One of the central goals of this paper is to try to determine which darts $X$ are reasonable, either 
against a given payoff function $f$ or a family of payoff functions.

We now wish to introduce a large collection of darts which turn out to be reasonable, and 
to do this we recall the notion of a \emph{selfdecomposable} probability measure
(see \cite{sato1999levy}). However, we first need to recall what it means for one random 
vector to divide another random vector.

\begin{definition} 
We say a random vector $X$ \textbf{divides} a random vector $Y$, written $X|Y$
if there exists a random vector $Z$ so that if $Z$ and $X$ are independent,
then $X+Z$ and $Y$ have the same distribution.
\end{definition}  

\begin{definition} 
A random vector $X$ is \textbf{selfdecomposable} if for all $d > 1$, $X|dX$.
\end{definition} 

Our first theorem, to be proved in Section \ref{section:selfdecomposible}, tells us that being selfdecomposable is a sufficient condition
for being reasonable.

\begin{theorem}
\label{theorem:reasonable.dart}
If $X|dX$, where $d>1$, then $g_{_{X,f}}(s)\geq g_{_{X,f}}(ds)$ for all $f$ and for all $s$. 
In particular, if $X$ is selfdecomposable, then $X$ is reasonable. 
\end{theorem}

The notion of selfdecomposability, while not as well known as other more standard
notions, was in fact already studied by Lévy, but not under this name. 
(Actually, the present authors came up with this notion in conjunction with this
project before learning that the concept already existed.) 

We mention that it is well known that all stable vectors are selfdecomposable (see 
\cite{sato1999levy}, p. 91) and that an independent sum of selfdecomposable 
random vectors is also selfdecomposable. Later on, we will give a large number of 
known examples of selfdecomposable distributions which are not stable.

While we have already proved that the uniform distribution is not reasonable,
we now study which darts $X$ are such that $(X,\cos(x))$ is reasonable in 1-dimension.
It turns out that by just using the payoff function
$\cos(x)$, we will be able to reveal that a number of distributions
are not reasonable. We will extend the following result to any dimension and also strengthen the statement
in Section \ref{section:cosine}.

\begin{theorem}\label{theorem:cfeasy}
Let $X$ be any dart taking values in $\mathbb{R}$ with characteristic function $\phi_X$.
Then $(X,\cos(x))$ is reasonable if and only if $|\phi_X(t)|$ is decreasing in $t$ on $[0,\infty)$.
\end{theorem}

\begin{corollary}\label{corollary:cf}
If $X$ is a 1-dimensional dart, whose characteristic function is analytic
and has a (real) zero, then $(X,\cos x)$ is not reasonable.
\end{corollary}

\begin{remark}
Analyticity in Corollary \ref{corollary:cf} is necessary since there is a symmetric dart 
whose characteristic function has a zero but such that $(X,\cos x)$ is reasonable. Namely,
it is known (see \cite{durrett2019probability}, p. 131) that if $X$ has density function
$\frac{1-\cos x}{\pi x^2}$, then its characteristic function is given by the tent function 
$\max\{1-|t|,0\}$. By Theorem \ref{theorem:cfeasy}, $(X,\cos x)$ is reasonable.
\end{remark}

Theorem \ref{theorem:cfeasy} gives us a powerful tool to study the behavior of $g_{_{X,f}}(d)$, and we can immediately find several examples of common distributions that are not reasonable with respect to 
$\cos(x)$. For example, the following distributions all have characteristic functions $\phi_X$ such that $|\phi_X(d)|$ is not 
decreasing for $d\in(0,\infty)$, and thus by Theorem \ref{theorem:cfeasy} 
are not reasonable with respect to 
$\cos(x)$:  Binomial distribution, Negative binomial distribution, Poisson distribution, Uniform distribution and Geometric distribution.  

We give two further examples of absolutely continuous distributions which are
not reasonable with respect to $\cos(x)$. These are the semi-circle distribution whose
probability density function on $[-1,1]$  is  \[\frac2\pi \cdot \sqrt{1-x^2}\]  
and the arcsine law whose probability density function on $[-1,1]$ is
\[\frac1{\pi\cdot \sqrt{(1+x)(1-x)}}. \]

For the semi-circle distribution, it is not hard to verify that the characteristic function is
positive at 3 and negative at 4 and hence must have a zero in between. By Corollary 
\ref{corollary:cf}, we conclude that the distribution function is not reasonable against $\cos(x)$.
It is interesting to also point out that this characteristic function is equal to 
$ \frac{2J_1(d)}{d}$  where $J_1$ is the so-called Bessel function of the first kind of order 1 
which is known to have its first zero at $\approx 3.8317$.

For the arcsine distribution,  the characteristic function is 
$J_0(d)$  where $J_0$ is the so-called Bessel function of the first kind of order 0 which is known
to have its first zero at $\approx 2.4$.
Therefore Corollary \ref{corollary:cf} again implies it is not reasonable against $\cos(x)$.

We note that the uniform distribution, the semi-circle distribution
and the arcsine distribution are all special cases of the symmetric Beta distribution. The characteristic
function of a Beta distribution is something which is called a confluent hypergeometric function.
It seems that the literature on confluent hypergeometric functions 
could supply answers to what happens for the general (symmetric) Beta distribution (with regard to
being reasonable against $\cos(x)$) but we have chosen not to investigate this. (In any case, none of the Beta distributions
will be reasonable by Theorem \ref{theorem:compactDart} since they are compactly supported.)

Finally, concerning the semi-circle distribution, since this is the projection onto the $x$-axis of the uniform distribution on the disc in 
$\mathbb{R}^2$, it follows from Proposition \ref{proposition:projections} that the latter distribution is not reasonable against $f(x,y)=\cos(x)$.
We will in fact see later on that this latter distribution is also not reasonable against the standard dart board.

We will now give two applications of Theorem \ref{theorem:cfeasy} which will also be proved in Section \ref{section:cosine}.

\begin{theorem}\label{theorem:2pointmasses}
If $X$ is a random variable with at least two point masses, then
$(X,\cos(x))$ is not reasonable. Moreover, if 
$X$ is a random vector with at least two point masses, then, for some $j$,
$(X,\cos(\pi_j(x)))$ is not reasonable where 
 $\pi_j$ is the projection onto the $j$th coordinate.
\end{theorem}

\begin{proposition} \label{proposition:cfnotToZero}
Let $X$ be a continuous dart (i.e., no point masses) taking values in $\mathbb{R}$, with 
characteristic 
function $\phi_X$. If $\phi_X(t)$ does not go to zero as $t\to\infty$,
then $(X,\cos(x))$ is not reasonable. (There is also some version of this in higher dimensions.)
\end{proposition} 

Note that by the Riemann-Lebesgue Lemma, any dart satisfying this assumption necessarily has
a nontrivial continuous singular component. However,
(see \cite{lyons1995seventy}) a continuous singular distribution may in fact have $\phi_X(t)$  
going to zero as $t\to\infty$. Such measures are called Rajchman measures and the first
example was constructed by Menshov (see \cite{menshov}).

While the payoff function $\cos(x)$ will in a number of cases reveal nonreasonableness,
it will not always succeed; i.e.\ there are nonreasonable darts $X$ so that 
$(X,\cos x)$ is reasonable.

For example, it is easy to check using Theorem \ref{theorem:cfeasy}
that a convex combination of a point mass at 0 and a normal
distribution is reasonable against cosine, but it will follow from 
Theorem \ref{theorem:aNonTrivialPointmass} below that it is not reasonable.

More interestingly, there are also absolutely continuous distributions with this property. 
In \cite{tlas2020bump}, a function $f\in C^\infty(\mathbb{R})$ which is real, 
nonnegative, symmetric, supported on $[-1,1]$, not identically equal to zero and 
such that its (real-valued) Fourier transform $\hat{f}(t)$ is monotone decreasing for $t\geq 0$
(and hence nonnegative) is constructed. After possibly rescaling, any such $f$ is the probability 
density function of some absolutely continuous random variable, which by Theorem 
$\ref{theorem:cfeasy}$ is reasonable with respect to $\cos(x)$. However it
is not reasonable according to Theorem \ref{theorem:compactDart} below. 

A third such example is covered by the next result which will also be proved in Section \ref{section:cosine}.
We know from either Theorem \ref{theorem:cfeasy} or Theorem \ref{theorem:2pointmasses} 
that a Bern($p$) distributed random variable with distribution $p\delta_1+ (1-p)\delta_0$
is not reasonable with respect to $\cos(x)$. One can ask if one could "smooth out" this distribution
so that it becomes reasonable with respect to $\cos(x)$. The following result yields a phase
transition where the answer to the question depends on both the parameter $p$ and the
degree of "smoothing out".  Note that it also yields
two darts $X,Y$ taking values in $\mathbb{R}$ such that $(X,\cos(x))$ is reasonable, 
$(Y,\cos(x))$ is not reasonable, but the independent sum of $X$ and $Y$ is such that 
$(X+Y,\cos(x))$ is reasonable.

\begin{theorem}
\label{theorem:reasPlusNonreasCosine}
Let $X_1$ be Bern($p$) distributed and $X_2$ be N(0,$\sigma^2$) distributed. If they are independent, then $X:=X_1+X_2$ is reasonable with respect to $\cos (x)$ if and only if
\begin{equation}
    \sigma^2d\Big(p^2+(1-p)^2+2(1-p)p\cos(d)\Big)+(1-p)p\sin(d)\geq 0,\ \forall d\geq0.
\end{equation}
 This implies that for $p=1/2$, $(X,\cos(x))$ is not reasonable for any $\sigma$, and that
 for any $p\neq 1/2$ there exists $\sigma_p \in (0,\infty)$ such that for all $\sigma\geq \sigma_p$, $(X,\cos(x))$ is reasonable, and for any $\sigma<\sigma_p$, $(X,\cos(x))$ is not reasonable. In addition, 
 for $p\neq 1/2$, $\sigma_p\leq (1-p)p/(\pi|1-2p|^2)$.
 
Finally for all $\sigma>0$  and $p\in (0,1)$,  $X$ is not reasonable. 
\end{theorem}

Our next four theorems identify further classes of nonreasonable darts.
The first two of these results will be proved in Section \ref{section:compactsupport}.

\begin{theorem}
\label{theorem:compactDart}
If $X$ is a nondegenerate dart in $\mathbb{R}$ which is compactly supported, then
$X$ is not reasonable with respect to some nonnegative 
continuous payoff function with compact support.
Furthermore, if $X$ is a dart in $\mathbb{R}^n$ for which some projection 
is nondegenerate and compactly supported, then $X$ is not reasonable with respect to some 
nonnegative continuous payoff  function with compact support.
\end{theorem}

The proof of this result will in fact prove the following result.

\begin{theorem}
\label{theorem:entire}
If $X$ is a nondegenerate dart in $\mathbb{R}$ whose Fourier transform in the
complex plane is entire and contains a zero, then $X$ is not reasonable against some 
continuous payoff function. Moreover, this is still true if the characteristic function is analytic at 0
and has a zero in its strip of regularity.
\end{theorem}

An illustrative example here are the family of probability density functions 
on $\mathbb{R}$ indexed by $\alpha>1$ given by 
$$
f_\alpha(x):= C_\alpha e^{-|x|^\alpha};
$$
the Fourier transforms of all of these are clearly entire since $\alpha>1$.
For $\alpha=2$, we have the normal density which is stable and hence selfdecomposable and therefore reasonable. However, for $\alpha\neq 2$,  P\'{o}lya (\cite{polya}) showed
that the Fourier transform has zeroes in the complex plane and hence 
Theorem \ref{theorem:entire} tells us that they are not reasonable.

Another illustrative example is $X=|Z|$ where $Z$ is a standard normal random variable. The Fourier transform of $X$ is (essentially)
the Mittag-Leffler function $E_{1/2}(z)$ which is known to have zeroes and hence
Theorem \ref{theorem:entire} tells us that $X$ is not reasonable. It is interesting to contrast this with 
$Y=|C|$ where $C$ is a standard Cauchy random variable which is known to be selfdecomposable 
and therefore reasonable. 

The next two results are proved in Sections \ref{section:pointmass} and  \ref{section:sing} respectively.

\begin{theorem} \label{theorem:aNonTrivialPointmass}
If $X$ is a nondegenerate dart taking values in $\mathbb{R}^n$ which has a single point mass,
then there exists a nonnegative
continuous payoff function $f$ with compact support such that $(X,f)$ is not reasonable.
\end{theorem}

\begin{theorem} \label{theorem:singular}
Given any dart $X$ in $\mathbb{R}$ which is not absolutely continuous and whose singular part
(in the Lebesgue decomposition) is compactly supported, 
then there exists a nonnegative bounded payoff function $f$ with compact support such that $(X,f)$ is not reasonable.
\end{theorem}

Note that in this last result, we only claim that $f$ is bounded, not that it is necessarily continuous.
This naturally then raises the question of whether the existence of $f$'s for which $(X,f)$ is not reasonable implies 
the existence of ``nice'' payoff functions $g$ for which $(X,g)$ is not reasonable.
Section \ref{section:improvement}  will provide a number of results of this type
but doesn't allow us to determine whether $f$ can be taken to be continuous in Theorem \ref{theorem:singular}.
One such result of this type is the following.

\begin{theorem}\label{theorem:absCont}
Assume that $X$ is a dart taking values in $\mathbb{R}^n$ with an absolutely continuous law 
$\mu_X$ and  $f$ is a bounded payoff function on $\mathbb{R}^n$ such that $(X,f)$ is not 
reasonable. Then there exists a nonnegative continuous payoff function $h$ with compact support on $\mathbb{R}^n$ such
that $(X,h)$ is not reasonable.
\end{theorem}
\begin{remark}
Without absolute continuity of $X$, we can still, by Proposition \ref{proposition:compactPayoff},
modify $f$ to be bounded and have compact support but not necessarily be continuous.
\end{remark}

Our next two results, which will be proved in Section \ref{section:closure},
concern operations which leave us within the class of reasonable distributions.

Let $\calF$ be some set of payoff functions, and let $\mathcal{X}_\calF$ 
be the set of darts which are reasonable with respect to $\calF$. It is natural to ask whether 
$X,Y\in\mathcal{X} _\calF$ implies that $X+Y\in\mathcal{X} _\calF$ ($X$ and $Y$ being
independent of course). The next result gives a positive answer for some classes $\calF$.

\begin{theorem} \label{theorem:closure}
Assume $\calF$ is a family of payoff functions and 
assume that $X_1,...,X_m$ are independent darts taking values in $\mathbb{R}^n$, each of 
which belongs to $\mathcal{X} _\calF$.  If $\calF$ is any of the following sets  
\begin{enumerate}
\item The set of all payoff functions
\item The set of all continuous payoff functions
\item The set of all bounded payoff functions
\item The set of all bounded continuous payoff functions
\item The set of all payoff functions of the same type as $\cos(\sum_{j=1}^n x_j)$,
\end{enumerate}
then for any $d_1,...,d_m,D_1,...,D_m\geq 0$ such that $d_j\leq D_j$ for all $j$ we have that 
\begin{equation}
    \sup_a Ef(a+\sum_{j=1}^m d_jX_j)
    \geq 
    \sup_a Ef(a+\sum_{j=1}^m D_jX_j),
    \ \forall f\in\calF{}
\end{equation}
and in particular $\sum_{j=1}^m X_j\in \mathcal{X} _\calF$. 
\end{theorem}

\medskip
Next, one would expect that the set of reasonable darts is closed with respect to 
convergence in distribution.  The following theorem provides a result of this type. 
See \eqref{eq:function.spaces} for some notation used here.

\begin{theorem} \label{theorem:convergenceInDistribution}
Let $\{X_j\}_{j=1}^\infty$ be a sequence of darts taking values in $\mathbb{R}^n$ which converges in 
distribution to some dart $X$. \\
(i) For any $f\in C_0(\mathbb{R}^n)$, $d> 0$, 
$\lim_j \g{X_j,f}(d)=g_{_{X,f}}(d)$. As a consequence, for any $f\in C_0(\mathbb{R}^n)$ for 
which $(X_j,f)$ is reasonable for all $j$, we have that $(X,f)$ is reasonable.  \\
(ii) The first statement in (i) is false if  $f\in C_0(\mathbb{R}^n)$ is replaced by 
$f\in C_b(\mathbb{R}^n)$. \\
(iii)  If, in addition, $X_j$ approaches $X$ in total variation, then 
(i) is still true if  $C_0(\mathbb{R}^n)$ is replaced by  $C_b(\mathbb{R}^n)$. \\
(iv) If each $X_j$ is reasonable with respect to $C_b(\mathbb{R}^n)$, then
 $X$ is reasonable with respect to $C_b(\mathbb{R}^n)$.  
\end{theorem}
Note that in (iv) we are not claiming that $(X_j,f)$ being reasonable for a fixed 
$f\in C_b(\mathbb{R}^n)$ implies that $(X,f)$ is reasonable.

\medskip
We are mostly concerned about whether darts are reasonable rather than whether payoff functions 
are reasonable. Nonetheless, we give one result of the latter type.

\begin{definition}
A function $f:\mathbb{R}^n\to\mathbb{R}$ is called \textbf{weakly unimodal} if for all 
$x\in \mathbb{R}^n$,  $f(rx)$ is decreasing in $r$ on $[0,\infty)$.
\end{definition}

In Section \ref{section:functions}, we prove 
\begin{proposition} \label{proposition:reasonable.function}
If $f$ is a weakly unimodal payoff function, then $f$ is reasonable.
\end{proposition}

The rest of the paper is organized as follows. In the second part of the present section, we discuss what happens with a {\it standard} dartboard.
In Section \ref{section:background}, we simply give some elementary background, some notation and a couple of 
elementary results. The theorems in the introduction are proved in the relevant sections as stated.
Finally, in Section \ref{section:open}, we list some questions. 

\subsection{What happens for {\it standard} darts?}

We have computed numerically what happens for the standard dartboard $f$ in $\mathbb{R}^2$ 
assuming that the dart distribution $X$  is uniform distribution on 
a disc.  We have looked at both the best place to aim and how the function $g_{_{X,f}}(d)$ behaves. We will see in particular that $(X,f)$ is not reasonable.

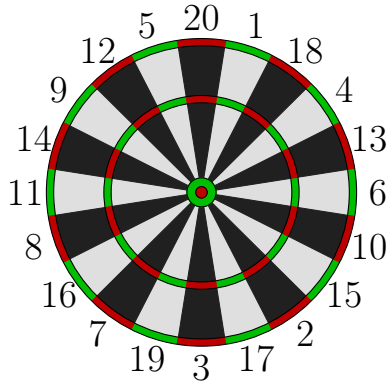
\begin{figure}    
\begin{center}
\begin{tikzpicture} [scale = 0.4]

\begin{scope}
\clip (0,0) circle (5.1cm);

\foreach \x in {0,2,4,6,8,10,12,14,16,18}
\filldraw[green!75!black] (0,0) -- ({6*cos(\x*18-9)}, {6*sin(\x*18-9)}) -- ({6*cos((\x+1)*18-9)}, {6*sin((\x+1)*18-9)}) -- cycle;

\foreach \x in {0,2,4,6,8,10,12,14,16,18}
\filldraw[red!75!black] (0,0) -- ({6*cos(\x*18+9)}, {6*sin(\x*18+9)}) -- ({6*cos((\x+1)*18+9)}, {6*sin((\x+1)*18+9)}) -- cycle;
\end{scope}

\begin{scope}
\clip (0,0) circle (4.86cm);

\foreach \x in {0,2,4,6,8,10,12,14,16,18}
\filldraw[white!75!gray] (0,0) -- ({6*cos(\x*18-9)}, {6*sin(\x*18-9)}) -- ({6*cos((\x+1)*18-9)}, {6*sin((\x+1)*18-9)}) -- cycle;

\foreach \x in {0,2,4,6,8,10,12,14,16,18}
\filldraw[gray!25!black] (0,0) -- ({6*cos(\x*18+9)}, {6*sin(\x*18+9)}) -- ({6*cos((\x+1)*18+9)}, {6*sin((\x+1)*18+9)}) -- cycle;
\end{scope}

\begin{scope}
\clip (0,0) circle (3.21cm);

\foreach \x in {0,2,4,6,8,10,12,14,16,18}
\filldraw[green!75!black] (0,0) -- ({6*cos(\x*18-9)}, {6*sin(\x*18-9)}) -- ({6*cos((\x+1)*18-9)}, {6*sin((\x+1)*18-9)}) -- cycle;

\foreach \x in {0,2,4,6,8,10,12,14,16,18}
\filldraw[red!75!black] (0,0) -- ({6*cos(\x*18+9)}, {6*sin(\x*18+9)}) -- ({6*cos((\x+1)*18+9)}, {6*sin((\x+1)*18+9)}) -- cycle;
\end{scope}

\begin{scope}
\clip (0,0) circle (2.97cm);

\foreach \x in {0,2,4,6,8,10,12,14,16,18}
\filldraw[white!75!gray] (0,0) -- ({6*cos(\x*18-9)}, {6*sin(\x*18-9)}) -- ({6*cos((\x+1)*18-9)}, {6*sin((\x+1)*18-9)}) -- cycle;

\foreach \x in {0,2,4,6,8,10,12,14,16,18}
\filldraw[gray!25!black] (0,0) -- ({6*cos(\x*18+9)}, {6*sin(\x*18+9)}) -- ({6*cos((\x+1)*18+9)}, {6*sin((\x+1)*18+9)}) -- cycle;
\end{scope}

\filldraw [green!75!black] (0,0) circle (0.48cm);
\filldraw [red!75!black] (0,0) circle (0.19cm);

\foreach \x in {0.19, 0.48,  2.97, 3.21,4.86, 5.1}
\draw (0,0) circle (\x cm);

\foreach \x/\y in {0/6,1/13,2/4,3/18,4/1, 5/20, 6/5, 7/12, 8/9 , 9/14, 10/11, 11/8, 12/16 , 13/7 , 14/19 , 15/3 , 16/17 , 17/2 , 18/15, 19/10}
\node at ({5.8*cos(\x*18)}, {5.8*sin(\x*18)})  {\Large \y}; 

\end{tikzpicture}

\caption{The dartboard is divided into 20 sectors with an irregular but standardized numbering. The black and white regions give a \emph{single} score (the plain number of the sector), while the outer red-green \emph{double ring} and the similar \emph{treble ring} half-way between the center and the rim give twice and three times the number of the sector respectively. In the middle of the board, the green \emph{bull's ring} gives 25 points, and the red \emph{bull's eye} 50 points. }
\label{F:dartboard}
\end{center}

\end{figure}

\begin{center}
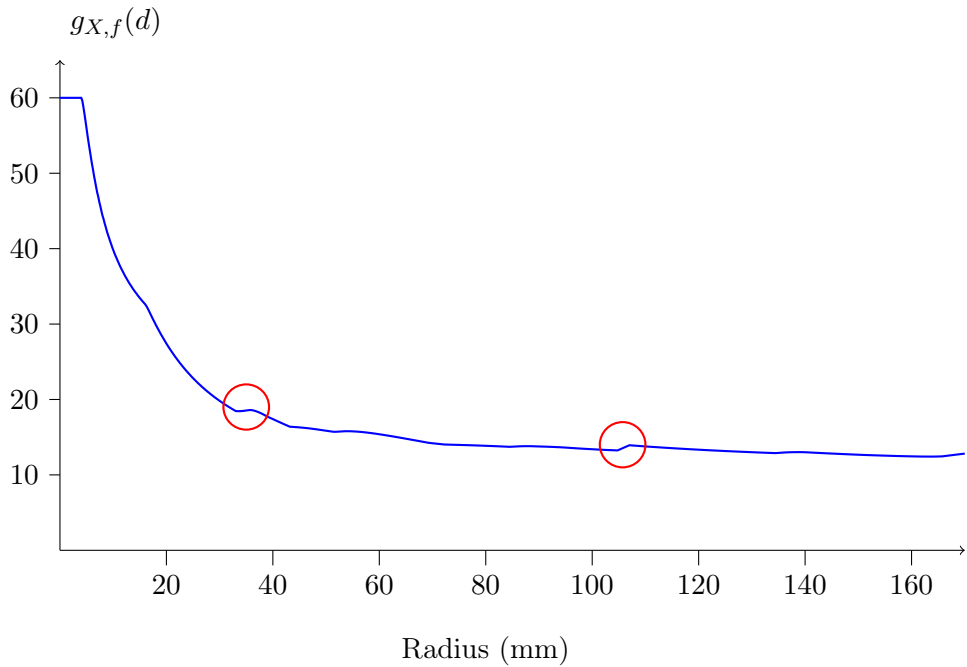
\begin{figure}
\begin{tikzpicture} 

\begin{scope} [xscale = 0.14]

\draw[->] (0,0) -- (0, 6.50) ; 
\draw [->] (0,0) -- (85.0, 0);

\node [right] at (0, 7.0) {$g_{X,f}(d)$};
\node [below] at (40.0, -1) {Radius (mm)};

\draw (-1, 1.00) -- (0, 1.00);
\node [left] at (-1, 1.00) {$10$};
\draw (-1, 2.00) -- (0, 2.00);
\node [left] at (-1, 2.00) {$20$};
\draw (-1, 3.00) -- (0, 3.00);
\node [left] at (-1, 3.00) {$30$};
\draw (-1, 4.00) -- (0, 4.00);
\node [left] at (-1, 4.00) {$40$};
\draw (-1, 5.00) -- (0, 5.00);
\node [left] at (-1, 5.00) {$50$};
\draw (-1, 6.00) -- (0, 6.00);
\node [left] at (-1, 6.00) {$60$};

\draw (10, -0.2) -- (10, 0);
\node [below] at (10, -0.2) {$20$};
\draw (20, -0.2) -- (20, 0);
\node [below] at (20, -0.2) {$40$};
\draw (30, -0.2) -- (30, 0);
\node [below] at (30, -0.2) {$60$};
\draw (40, -0.2) -- (40, 0);
\node [below] at (40, -0.2) {$80$};
\draw (50, -0.2) -- (50, 0);
\node [below] at (50, -0.2) {$100$};
\draw (60, -0.2) -- (60, 0);
\node [below] at (60, -0.2) {$120$};
\draw (70, -0.2) -- (70, 0);
\node [below] at (70, -0.2) {$140$};
\draw (80, -0.2) -- (80, 0);
\node [below] at (80, -0.2) {$160$};

\draw [thick, blue] (0.0, 6.0) -- (0.1, 6.0) -- (0.2, 6.0) -- (0.3, 6.0) -- (0.4, 6.0) -- (0.5, 6.0) -- (0.6, 6.0) -- (0.7, 6.0) -- (0.8, 6.0) -- (0.90000004, 6.0) -- (1.0, 6.0) -- (1.1, 6.0) -- (1.2, 6.0) -- (1.3000001, 6.0) -- (1.4, 6.0) -- (1.5, 6.0) -- (1.6, 6.0) -- (1.7, 6.0) -- (1.8000001, 6.0) -- (1.9, 6.0) -- (2.0, 6.0) -- (2.1000001, 5.9621267) -- (2.2, 5.881345) -- (2.3, 5.789474) -- (2.4, 5.6921363) -- (2.5, 5.592045) -- (2.6000001, 5.490806) -- (2.7, 5.3971167) -- (2.8, 5.306971) -- (2.9, 5.2179537) -- (3.0, 5.1279693) -- (3.1000001, 5.052316) -- (3.2, 4.9654098) -- (3.3, 4.892344) -- (3.4, 4.816) -- (3.5, 4.742798) -- (3.6000001, 4.686405) -- (3.7, 4.616352) -- (3.8, 4.5597167) -- (3.9, 4.494453) -- (4.0, 4.436617) -- (4.1, 4.3911805) -- (4.2000003, 4.3348417) -- (4.3, 4.2870965) -- (4.4, 4.2346554) -- (4.5, 4.1877065) -- (4.6, 4.14883) -- (4.7000003, 4.10432) -- (4.8, 4.0646057) -- (4.9, 4.02206) -- (5.0, 3.9844487) -- (5.1, 3.946409) -- (5.2000003, 3.910792) -- (5.3, 3.8798954) -- (5.4, 3.8458176) -- (5.5, 3.8149204) -- (5.6, 3.781615) -- (5.7000003, 3.7528706) -- (5.8, 3.7239747) -- (5.9, 3.6974251) -- (6.0, 3.66959) -- (6.1, 3.6423252) -- (6.2000003, 3.618108) -- (6.3, 3.5931904) -- (6.4, 3.568817) -- (6.5, 3.5447903) -- (6.6, 3.5230014) -- (6.7000003, 3.5024514) -- (6.8, 3.4800415) -- (6.9, 3.458559) -- (7.0, 3.4391465) -- (7.1, 3.4193387) -- (7.2000003, 3.4021306) -- (7.3, 3.3813138) -- (7.4, 3.363578) -- (7.5, 3.345712) -- (7.6, 3.3292139) -- (7.7000003, 3.3121204) -- (7.8, 3.2942593) -- (7.9, 3.279665) -- (8.0, 3.263483) -- (8.1, 3.2430825) -- (8.2, 3.219146) -- (8.3, 3.1911876) -- (8.400001, 3.1638188) -- (8.5, 3.1329722) -- (8.6, 3.1048326) -- (8.7, 3.0755944) -- (8.8, 3.0474107) -- (8.900001, 3.020471) -- (9.0, 2.991409) -- (9.1, 2.9651344) -- (9.2, 2.9387276) -- (9.3, 2.9127648) -- (9.400001, 2.8867648) -- (9.5, 2.8610408) -- (9.6, 2.837241) -- (9.7, 2.8128433) -- (9.8, 2.7885206) -- (9.900001, 2.765156) -- (10.0, 2.7411878) -- (10.1, 2.7192805) -- (10.2, 2.6965363) -- (10.3, 2.6749768) -- (10.400001, 2.6540694) -- (10.5, 2.63298) -- (10.6, 2.6124089) -- (10.7, 2.5919561) -- (10.8, 2.5719645) -- (10.900001, 2.553058) -- (11.0, 2.5339592) -- (11.1, 2.5147045) -- (11.2, 2.496567) -- (11.3, 2.4778118) -- (11.400001, 2.4608095) -- (11.5, 2.442169) -- (11.6, 2.4252408) -- (11.7, 2.4086854) -- (11.8, 2.391995) -- (11.900001, 2.37561) -- (12.0, 2.359383) -- (12.1, 2.3439949) -- (12.2, 2.3288722) -- (12.3, 2.3132863) -- (12.400001, 2.2977517) -- (12.5, 2.2830777) -- (12.6, 2.268711) -- (12.7, 2.2551553) -- (12.8, 2.240684) -- (12.900001, 2.226762) -- (13.0, 2.2131188) -- (13.1, 2.199833) -- (13.2, 2.1870184) -- (13.3, 2.173501) -- (13.400001, 2.1607597) -- (13.5, 2.148912) -- (13.6, 2.1360328) -- (13.7, 2.1239343) -- (13.8, 2.1120677) -- (13.900001, 2.100361) -- (14.0, 2.0889986) -- (14.1, 2.0773036) -- (14.2, 2.0661244) -- (14.3, 2.0550048) -- (14.400001, 2.0442665) -- (14.5, 2.0330079) -- (14.6, 2.0224562) -- (14.7, 2.0122626) -- (14.8, 2.0020587) -- (14.900001, 1.9916893) -- (15.0, 1.9815664) -- (15.1, 1.9719046) -- (15.2, 1.9625301) -- (15.3, 1.9526232) -- (15.400001, 1.9434963) -- (15.5, 1.9337481) -- (15.6, 1.9247707) -- (15.7, 1.9158577) -- (15.8, 1.9066035) -- (15.900001, 1.8981279) -- (16.0, 1.8893496) -- (16.1, 1.8809369) -- (16.2, 1.8721253) -- (16.300001, 1.8637766) -- (16.4, 1.8558174) -- (16.5, 1.8474369) -- (16.6, 1.8458111) -- (16.7, 1.845026) -- (16.800001, 1.8449291) -- (16.9, 1.8452778) -- (17.0, 1.845632) -- (17.1, 1.8469404) -- (17.2, 1.8481245) -- (17.300001, 1.8499825) -- (17.4, 1.8518156) -- (17.5, 1.8537892) -- (17.6, 1.8559735) -- (17.7, 1.857712) -- (17.800001, 1.8599154) -- (17.9, 1.8598593) -- (18.0, 1.8588799) -- (18.1, 1.8566135) -- (18.2, 1.8530912) -- (18.300001, 1.8495034) -- (18.4, 1.8448724) -- (18.5, 1.8400973) -- (18.6, 1.8348703) -- (18.7, 1.8289973) -- (18.800001, 1.8228289) -- (18.9, 1.816197) -- (19.0, 1.8091313) -- (19.1, 1.8021475) -- (19.2, 1.7946062) -- (19.300001, 1.7872151) -- (19.4, 1.7793449) -- (19.5, 1.7715296) -- (19.6, 1.7652066) -- (19.7, 1.7589377) -- (19.800001, 1.7527936) -- (19.9, 1.7463881) -- (20.0, 1.7398818) -- (20.1, 1.7334839) -- (20.2, 1.7268914) -- (20.300001, 1.7205924) -- (20.4, 1.7138271) -- (20.5, 1.70726) -- (20.6, 1.7009977) -- (20.7, 1.6946421) -- (20.800001, 1.6883593) -- (20.9, 1.6819044) -- (21.0, 1.6755279) -- (21.1, 1.669479) -- (21.2, 1.6630553) -- (21.300001, 1.6570463) -- (21.4, 1.6508055) -- (21.5, 1.644648) -- (21.6, 1.638798) -- (21.7, 1.6379608) -- (21.800001, 1.6369826) -- (21.9, 1.6361637) -- (22.0, 1.6353226) -- (22.1, 1.6341285) -- (22.2, 1.6331669) -- (22.300001, 1.6320333) -- (22.4, 1.6310265) -- (22.5, 1.6297423) -- (22.6, 1.6284714) -- (22.7, 1.6272119) -- (22.800001, 1.6259634) -- (22.9, 1.6245251) -- (23.0, 1.6230907) -- (23.1, 1.6216611) -- (23.2, 1.6202374) -- (23.300001, 1.6186609) -- (23.4, 1.6169271) -- (23.5, 1.6152285) -- (23.6, 1.6135771) -- (23.7, 1.6117958) -- (23.800001, 1.6098776) -- (23.9, 1.6080809) -- (24.0, 1.6061525) -- (24.1, 1.6042938) -- (24.2, 1.6022122) -- (24.300001, 1.6001428) -- (24.4, 1.598121) -- (24.5, 1.5960305) -- (24.6, 1.5939885) -- (24.7, 1.5918391) -- (24.800001, 1.5897979) -- (24.9, 1.5877669) -- (25.0, 1.5857061) -- (25.1, 1.5837649) -- (25.2, 1.5816916) -- (25.300001, 1.5797328) -- (25.4, 1.57772) -- (25.5, 1.5756874) -- (25.6, 1.5737435) -- (25.7, 1.5718126) -- (25.800001, 1.5711814) -- (25.9, 1.5718262) -- (26.0, 1.5725284) -- (26.1, 1.5732467) -- (26.2, 1.5740354) -- (26.300001, 1.5748565) -- (26.4, 1.5756552) -- (26.5, 1.5765451) -- (26.6, 1.5774857) -- (26.7, 1.578404) -- (26.800001, 1.5789855) -- (26.9, 1.5791998) -- (27.0, 1.5792354) -- (27.1, 1.5790681) -- (27.2, 1.5787714) -- (27.300001, 1.5783198) -- (27.4, 1.5778174) -- (27.5, 1.5772117) -- (27.6, 1.576477) -- (27.7, 1.5757341) -- (27.800001, 1.5748324) -- (27.9, 1.5739776) -- (28.0, 1.5729624) -- (28.1, 1.5717921) -- (28.2, 1.5706767) -- (28.300001, 1.5693694) -- (28.4, 1.5680825) -- (28.5, 1.5665944) -- (28.6, 1.5650886) -- (28.7, 1.5636116) -- (28.800001, 1.5619261) -- (28.9, 1.5602602) -- (29.0, 1.5584873) -- (29.1, 1.5567087) -- (29.2, 1.5550193) -- (29.300001, 1.5530394) -- (29.4, 1.5511832) -- (29.5, 1.5492191) -- (29.6, 1.5472822) -- (29.7, 1.5452689) -- (29.800001, 1.5431536) -- (29.9, 1.5411353) -- (30.0, 1.538988) -- (30.1, 1.5368812) -- (30.2, 1.5347142) -- (30.300001, 1.5324788) -- (30.4, 1.5303224) -- (30.5, 1.5280348) -- (30.6, 1.5257707) -- (30.7, 1.5235498) -- (30.800001, 1.5212263) -- (30.9, 1.5189648) -- (31.0, 1.5164813) -- (31.1, 1.5141951) -- (31.2, 1.5117998) -- (31.300001, 1.5094572) -- (31.4, 1.5070615) -- (31.5, 1.5045689) -- (31.6, 1.5021526) -- (31.7, 1.4997301) -- (31.800001, 1.4972823) -- (31.9, 1.4947755) -- (32.0, 1.4923409) -- (32.100002, 1.4898999) -- (32.2, 1.4874014) -- (32.3, 1.4848348) -- (32.4, 1.4823221) -- (32.5, 1.4798108) -- (32.600002, 1.4773406) -- (32.7, 1.4747974) -- (32.8, 1.4722437) -- (32.9, 1.4696716) -- (33.0, 1.4671714) -- (33.100002, 1.464615) -- (33.2, 1.462133) -- (33.3, 1.4595809) -- (33.4, 1.4569381) -- (33.5, 1.4544014) -- (33.600002, 1.4518328) -- (33.7, 1.4492761) -- (33.8, 1.4467858) -- (33.9, 1.4441307) -- (34.0, 1.441547) -- (34.100002, 1.4389431) -- (34.2, 1.4363991) -- (34.3, 1.4338789) -- (34.4, 1.4312133) -- (34.5, 1.4287283) -- (34.600002, 1.4260527) -- (34.7, 1.4242443) -- (34.8, 1.4226642) -- (34.9, 1.4210936) -- (35.0, 1.4196044) -- (35.100002, 1.4180738) -- (35.2, 1.4165686) -- (35.3, 1.4150478) -- (35.4, 1.4135526) -- (35.5, 1.4120847) -- (35.600002, 1.4105419) -- (35.7, 1.4090759) -- (35.8, 1.4075608) -- (35.9, 1.4061031) -- (36.0, 1.4046422) -- (36.100002, 1.403138) -- (36.2, 1.4025986) -- (36.3, 1.402227) -- (36.4, 1.4017475) -- (36.5, 1.4013493) -- (36.600002, 1.4009365) -- (36.7, 1.4005082) -- (36.8, 1.400123) -- (36.9, 1.3997316) -- (37.0, 1.3993701) -- (37.100002, 1.3989623) -- (37.2, 1.3985934) -- (37.3, 1.3982569) -- (37.4, 1.3978628) -- (37.5, 1.3975195) -- (37.600002, 1.3970995) -- (37.7, 1.3966914) -- (37.8, 1.3963288) -- (37.9, 1.3959202) -- (38.0, 1.3954954) -- (38.100002, 1.3950489) -- (38.2, 1.3946474) -- (38.3, 1.3941936) -- (38.4, 1.393744) -- (38.5, 1.3932564) -- (38.600002, 1.392828) -- (38.7, 1.392361) -- (38.8, 1.3918829) -- (38.9, 1.3913823) -- (39.0, 1.3909029) -- (39.100002, 1.3904314) -- (39.2, 1.3899168) -- (39.3, 1.3893946) -- (39.4, 1.3888782) -- (39.5, 1.3883923) -- (39.600002, 1.3878777) -- (39.7, 1.3873241) -- (39.8, 1.3867596) -- (39.9, 1.3862904) -- (40.0, 1.3857378) -- (40.100002, 1.3851843) -- (40.2, 1.3845829) -- (40.3, 1.3841047) -- (40.4, 1.3835491) -- (40.5, 1.3829912) -- (40.600002, 1.382444) -- (40.7, 1.3818866) -- (40.8, 1.3813446) -- (40.9, 1.3807561) -- (41.0, 1.3801849) -- (41.100002, 1.3796057) -- (41.2, 1.3790268) -- (41.3, 1.3784341) -- (41.4, 1.3778071) -- (41.5, 1.3771367) -- (41.600002, 1.3765395) -- (41.7, 1.3758817) -- (41.8, 1.375233) -- (41.9, 1.374564) -- (42.0, 1.3738922) -- (42.100002, 1.3732136) -- (42.2, 1.3725085) -- (42.3, 1.3725936) -- (42.4, 1.3733519) -- (42.5, 1.3741536) -- (42.600002, 1.3748757) -- (42.7, 1.3756611) -- (42.8, 1.3764514) -- (42.9, 1.3771814) -- (43.0, 1.3778768) -- (43.100002, 1.3785001) -- (43.2, 1.379054) -- (43.3, 1.3795592) -- (43.4, 1.3799738) -- (43.5, 1.3803325) -- (43.600002, 1.3805512) -- (43.7, 1.3807348) -- (43.8, 1.3808668) -- (43.9, 1.3808775) -- (44.0, 1.3808509) -- (44.100002, 1.380742) -- (44.2, 1.380534) -- (44.3, 1.3802559) -- (44.4, 1.3798537) -- (44.5, 1.3793696) -- (44.600002, 1.3788875) -- (44.7, 1.3783973) -- (44.8, 1.3779298) -- (44.9, 1.3774676) -- (45.0, 1.3770174) -- (45.100002, 1.3765639) -- (45.2, 1.3761486) -- (45.3, 1.3757437) -- (45.4, 1.3753219) -- (45.5, 1.374895) -- (45.600002, 1.3744572) -- (45.7, 1.3740023) -- (45.8, 1.3735098) -- (45.9, 1.3730174) -- (46.0, 1.3724836) -- (46.100002, 1.3719414) -- (46.2, 1.3713694) -- (46.3, 1.370757) -- (46.4, 1.3701428) -- (46.5, 1.3695207) -- (46.600002, 1.3688761) -- (46.7, 1.3682396) -- (46.8, 1.3675791) -- (46.9, 1.3669109) -- (47.0, 1.3661989) -- (47.100002, 1.3654703) -- (47.2, 1.3647102) -- (47.3, 1.3639225) -- (47.4, 1.3630615) -- (47.5, 1.3621601) -- (47.600002, 1.3611673) -- (47.7, 1.360215) -- (47.8, 1.3592762) -- (47.9, 1.3583055) -- (48.0, 1.3573684) -- (48.100002, 1.3564041) -- (48.2, 1.3554698) -- (48.3, 1.3545142) -- (48.4, 1.3535844) -- (48.5, 1.3526499) -- (48.600002, 1.3517369) -- (48.7, 1.3508571) -- (48.8, 1.3499678) -- (48.9, 1.3490952) -- (49.0, 1.3482392) -- (49.100002, 1.3474295) -- (49.2, 1.3465866) -- (49.3, 1.3457913) -- (49.4, 1.3449957) -- (49.5, 1.3442082) -- (49.600002, 1.3434213) -- (49.7, 1.3426384) -- (49.8, 1.3418727) -- (49.9, 1.3410919) -- (50.0, 1.340326) -- (50.100002, 1.3395662) -- (50.2, 1.3387843) -- (50.3, 1.338056) -- (50.4, 1.3372833) -- (50.5, 1.3365347) -- (50.600002, 1.3357819) -- (50.7, 1.3350037) -- (50.8, 1.3342818) -- (50.9, 1.3335341) -- (51.0, 1.332827) -- (51.100002, 1.3320942) -- (51.2, 1.3313576) -- (51.3, 1.3306619) -- (51.4, 1.3301152) -- (51.5, 1.3295398) -- (51.600002, 1.3289961) -- (51.7, 1.3284357) -- (51.8, 1.3278728) -- (51.9, 1.3272771) -- (52.0, 1.3266652) -- (52.100002, 1.3260568) -- (52.2, 1.3254577) -- (52.3, 1.324851) -- (52.4, 1.3253603) -- (52.5, 1.3316854) -- (52.600002, 1.3380296) -- (52.7, 1.3442879) -- (52.8, 1.3504696) -- (52.9, 1.3566768) -- (53.0, 1.3629198) -- (53.100002, 1.3690106) -- (53.2, 1.3751527) -- (53.3, 1.3811355) -- (53.4, 1.3872403) -- (53.5, 1.3931829) -- (53.600002, 1.3921008) -- (53.7, 1.3909962) -- (53.8, 1.3899468) -- (53.9, 1.3888474) -- (54.0, 1.3878181) -- (54.100002, 1.3867658) -- (54.2, 1.3857317) -- (54.3, 1.3846822) -- (54.4, 1.3836282) -- (54.5, 1.3826073) -- (54.600002, 1.381587) -- (54.7, 1.3805572) -- (54.8, 1.3795675) -- (54.9, 1.3785415) -- (55.0, 1.3775566) -- (55.100002, 1.3765447) -- (55.2, 1.3755543) -- (55.3, 1.3745939) -- (55.4, 1.3735942) -- (55.5, 1.3726374) -- (55.600002, 1.371653) -- (55.7, 1.3707072) -- (55.8, 1.3697724) -- (55.9, 1.3688188) -- (56.0, 1.3678793) -- (56.100002, 1.3669143) -- (56.2, 1.3659949) -- (56.3, 1.365091) -- (56.4, 1.364151) -- (56.5, 1.3632315) -- (56.600002, 1.3623267) -- (56.7, 1.3614157) -- (56.8, 1.3605146) -- (56.9, 1.3596121) -- (57.0, 1.3587174) -- (57.100002, 1.3578498) -- (57.2, 1.3569566) -- (57.3, 1.35609) -- (57.4, 1.3552125) -- (57.5, 1.3543468) -- (57.600002, 1.3535006) -- (57.7, 1.3526254) -- (57.8, 1.3517698) -- (57.9, 1.3509252) -- (58.0, 1.35007) -- (58.100002, 1.3492275) -- (58.2, 1.3483869) -- (58.3, 1.347565) -- (58.4, 1.3467475) -- (58.5, 1.3459039) -- (58.600002, 1.3451033) -- (58.7, 1.3442873) -- (58.8, 1.3434981) -- (58.9, 1.342693) -- (59.0, 1.3418812) -- (59.100002, 1.3411028) -- (59.2, 1.3403058) -- (59.3, 1.339531) -- (59.4, 1.3387375) -- (59.5, 1.3379495) -- (59.600002, 1.3371952) -- (59.7, 1.3364197) -- (59.8, 1.3356351) -- (59.9, 1.3348788) -- (60.0, 1.3341246) -- (60.100002, 1.3333921) -- (60.2, 1.3326203) -- (60.3, 1.3318775) -- (60.4, 1.3311399) -- (60.5, 1.3303984) -- (60.600002, 1.3296739) -- (60.7, 1.328929) -- (60.8, 1.3282032) -- (60.9, 1.3274869) -- (61.0, 1.3267817) -- (61.100002, 1.3260556) -- (61.2, 1.3253481) -- (61.3, 1.3246473) -- (61.4, 1.3239499) -- (61.5, 1.323239) -- (61.600002, 1.3225411) -- (61.7, 1.3218564) -- (61.8, 1.3211755) -- (61.9, 1.32048) -- (62.0, 1.3198009) -- (62.100002, 1.31912) -- (62.2, 1.3184584) -- (62.3, 1.317779) -- (62.4, 1.3171109) -- (62.5, 1.3164507) -- (62.600002, 1.3157867) -- (62.7, 1.3151274) -- (62.8, 1.3144679) -- (62.9, 1.3138195) -- (63.0, 1.3131846) -- (63.100002, 1.3125373) -- (63.2, 1.3118876) -- (63.3, 1.3112506) -- (63.4, 1.3106197) -- (63.5, 1.309997) -- (63.600002, 1.3093585) -- (63.7, 1.3087325) -- (63.8, 1.30811) -- (63.9, 1.3075007) -- (64.0, 1.3068775) -- (64.1, 1.3062619) -- (64.200005, 1.3056667) -- (64.3, 1.3050553) -- (64.4, 1.3044548) -- (64.5, 1.3038456) -- (64.6, 1.3032484) -- (64.700005, 1.3026671) -- (64.8, 1.3020784) -- (64.9, 1.3014843) -- (65.0, 1.3008851) -- (65.1, 1.3003136) -- (65.200005, 1.2997426) -- (65.3, 1.2991427) -- (65.4, 1.2985831) -- (65.5, 1.2980034) -- (65.6, 1.29744) -- (65.700005, 1.296873) -- (65.8, 1.2963012) -- (65.9, 1.2957554) -- (66.0, 1.2951952) -- (66.1, 1.2946357) -- (66.200005, 1.2940813) -- (66.3, 1.2935269) -- (66.4, 1.2929901) -- (66.5, 1.2924427) -- (66.6, 1.2918996) -- (66.700005, 1.2913551) -- (66.8, 1.2908219) -- (66.9, 1.2902945) -- (67.0, 1.2897507) -- (67.1, 1.2892207) -- (67.200005, 1.2889885) -- (67.3, 1.2896656) -- (67.4, 1.2904855) -- (67.5, 1.2914795) -- (67.6, 1.2925829) -- (67.700005, 1.293698) -- (67.8, 1.2947314) -- (67.9, 1.2956865) -- (68.0, 1.296563) -- (68.1, 1.2973607) -- (68.200005, 1.2980765) -- (68.3, 1.2987392) -- (68.4, 1.2993288) -- (68.5, 1.2998478) -- (68.6, 1.3002881) -- (68.700005, 1.3006307) -- (68.8, 1.3009144) -- (68.9, 1.3011869) -- (69.0, 1.3014145) -- (69.1, 1.301586) -- (69.200005, 1.3017114) -- (69.3, 1.3018155) -- (69.4, 1.3018769) -- (69.5, 1.3018637) -- (69.6, 1.3017566) -- (69.700005, 1.3015616) -- (69.8, 1.301284) -- (69.9, 1.3008918) -- (70.0, 1.3003957) -- (70.1, 1.2997826) -- (70.200005, 1.2990686) -- (70.3, 1.2983389) -- (70.4, 1.2975988) -- (70.5, 1.296872) -- (70.6, 1.2961563) -- (70.700005, 1.2954315) -- (70.8, 1.294718) -- (70.9, 1.293985) -- (71.0, 1.2932806) -- (71.1, 1.2925818) -- (71.200005, 1.2918642) -- (71.3, 1.2911795) -- (71.4, 1.2904761) -- (71.5, 1.2897903) -- (71.6, 1.2890862) -- (71.700005, 1.2883877) -- (71.8, 1.2877139) -- (71.9, 1.2870258) -- (72.0, 1.286363) -- (72.1, 1.2856804) -- (72.200005, 1.2850055) -- (72.3, 1.284336) -- (72.4, 1.283655) -- (72.5, 1.2829944) -- (72.6, 1.2823364) -- (72.700005, 1.2816885) -- (72.8, 1.281026) -- (72.9, 1.2803615) -- (73.0, 1.2797197) -- (73.1, 1.2790748) -- (73.200005, 1.2784413) -- (73.3, 1.2777995) -- (73.4, 1.2771586) -- (73.5, 1.2765268) -- (73.6, 1.2758952) -- (73.700005, 1.275283) -- (73.8, 1.2746469) -- (73.9, 1.2740366) -- (74.0, 1.273398) -- (74.1, 1.2727784) -- (74.200005, 1.2721761) -- (74.3, 1.2715629) -- (74.4, 1.2709621) -- (74.5, 1.2703584) -- (74.6, 1.2697554) -- (74.700005, 1.2691662) -- (74.8, 1.2685721) -- (74.9, 1.2679979) -- (75.0, 1.2674035) -- (75.1, 1.2668413) -- (75.200005, 1.2662822) -- (75.3, 1.2657145) -- (75.4, 1.2651883) -- (75.5, 1.2646241) -- (75.6, 1.2641244) -- (75.700005, 1.263624) -- (75.8, 1.2631396) -- (75.9, 1.2626381) -- (76.0, 1.2621547) -- (76.1, 1.2616569) -- (76.200005, 1.2611746) -- (76.3, 1.2606977) -- (76.4, 1.2602204) -- (76.5, 1.2597516) -- (76.6, 1.259273) -- (76.700005, 1.2588058) -- (76.8, 1.2583364) -- (76.9, 1.2578787) -- (77.0, 1.2574291) -- (77.1, 1.2569672) -- (77.200005, 1.2565131) -- (77.3, 1.2560657) -- (77.4, 1.2556204) -- (77.5, 1.2551807) -- (77.6, 1.2547405) -- (77.700005, 1.2543014) -- (77.8, 1.2538657) -- (77.9, 1.2534432) -- (78.0, 1.2530159) -- (78.1, 1.2525934) -- (78.200005, 1.2521783) -- (78.3, 1.251759) -- (78.4, 1.2513593) -- (78.5, 1.2509574) -- (78.6, 1.250566) -- (78.700005, 1.250188) -- (78.8, 1.2497997) -- (78.9, 1.2494254) -- (79.0, 1.2490588) -- (79.1, 1.2487087) -- (79.200005, 1.2483615) -- (79.3, 1.2480155) -- (79.4, 1.2476866) -- (79.5, 1.2473592) -- (79.6, 1.2470424) -- (79.700005, 1.2467307) -- (79.8, 1.246426) -- (79.9, 1.2461361) -- (80.0, 1.245831) -- (80.1, 1.2455543) -- (80.200005, 1.2452581) -- (80.3, 1.2449831) -- (80.4, 1.244704) -- (80.5, 1.2444259) -- (80.6, 1.244167) -- (80.700005, 1.2438928) -- (80.8, 1.2436363) -- (80.9, 1.2433844) -- (81.0, 1.2432216) -- (81.1, 1.2431315) -- (81.200005, 1.2430354) -- (81.3, 1.2429616) -- (81.4, 1.2428991) -- (81.5, 1.2428435) -- (81.6, 1.2428055) -- (81.700005, 1.2427466) -- (81.8, 1.2427217) -- (81.9, 1.2426938) -- (82.0, 1.2426757) -- (82.1, 1.2428042) -- (82.200005, 1.2431146) -- (82.3, 1.2434275) -- (82.4, 1.2437632) -- (82.5, 1.2440889) -- (82.6, 1.2444426) -- (82.700005, 1.2448542) -- (82.8, 1.2454506) -- (82.9, 1.2463874) -- (83.0, 1.2476128) -- (83.1, 1.2493535) -- (83.200005, 1.251142) -- (83.3, 1.2529289) -- (83.4, 1.2547046) -- (83.5, 1.2564876) -- (83.6, 1.2582597) -- (83.700005, 1.2600158) -- (83.8, 1.2617561) -- (83.9, 1.2635309) -- (84.0, 1.2652713) -- (84.1, 1.2670195) -- (84.200005, 1.2687318) -- (84.3, 1.2704527) -- (84.4, 1.2721906) -- (84.5, 1.2739052) -- (84.6, 1.2756032) -- (84.700005, 1.2772673) -- (84.8, 1.2789927) -- (84.9, 1.2806818) -- (85.0, 1.2823628);

\end{scope}

\draw [thick, red] (2.45, 1.9) circle (0.3);
\draw [thick, red] (7.4, 1.4) circle (0.3);

\end{tikzpicture}

\caption{
Two of the at least seven places where this function increases are circled in red.}
\label{F:diagram}
\end{figure}
\end{center}

When the disc is small enough to fit completely inside the treble 20 region, this is of course where to aim, and the expected payoff is 60 points. 
When the radius increases beyond 4 mm this is no longer possible, but the best place to aim is still at treble 20. At radius about 16 mm, we can no longer keep the 
disc completely inside the 20 sector while at the same time maximizing its intersection with the treble ring. Here the best place to aim is somewhere in the ``fat 20'', 
a bit above the treble 20, making a compromise between including as much as possible of the treble 20 region while not having too much of the disc 
sticking out into the 5- and 1-sectors.   

At radius 33.6 mm, something interesting happens: The average score starts to increase as the radius of the disc increases further! This is because as 
the radius goes from about 33.6 to 35.7 mm, we can suddenly fit most of both treble 20 \emph{and} the double 20 into the disc. At radius 33.6 mm we 
score an average of 18.45 points, but with the larger radius 35.7 mm, the score increases to 18.60!

A little later, at radius 39 mm, the best place to aim suddenly jumps from the 20-sector to the 19-sector. At 43 mm there is another jump after which we 
should aim near treble 7, with fair chances of scoring 16 or 19 or even a high treble. 

As the radius increases, there are (at least!) five more places where $g_{_{X,f}}(d)$ increases. The most clear-cut of them 
is when the radius increases from 104.8 mm to 107 mm, which is where we can fit the entire treble ring into the disc. The last one occurs from radius 164 mm, 
when we only get half the double ring into the disc, to 170 mm, which is the radius of the whole dartboard. 

Lastly, we mention that there have been some studies for the standard dartboard with the dart being the normal distribution concerning what 
$g_{_{X,f}}(d)$ looks like and how one should play optimally (i.e., where one should aim). See \cite{TPT}  and  \url{http://datagenetics.com/blog/january12012/index.html}

\section{Background, Notation and some elementary results}  \label{section:background} 
\subsection{Various background}

First, the term \textbf{nondegenerate} will refer to any distribution which is not a single point mass.

\begin{definition}
We say that the random vectors $X$ and $Y$ have the same \textbf{type} if there exist 
$a>0$ and $b\in\mathbb{R}^n$ such that $Y$ and $aX+b$ have the same distribution.
We say that the functions $f$ and $h$ have the same \textbf{type} if there exist 
$a,c>0$, $b\in\mathbb{R}^n$, and $d\in\mathbb{R}$ such that
$h(x)=cf(ax+b)+d$ for all $x$. 
\end{definition}

If  $X$ and $Y$ have the same type and $f$ and $h$ have the same type,
then it will follow from Proposition \ref{proposition:basic_g_properties} below that $(X,f)$ is 
reasonable if and only if $(Y,h)$ is.

\begin{remark}
It is easy to check that $X|dX$ only depends on $X$'s type
and hence being selfdecomposable also only depends on the type of $X$.
\end{remark}

\begin{definition}
Let $X$ be a random vector taking values in $\mathbb{R}^n$. The \textbf{characteristic function} 
of $X$ is the function $\phi_X:\mathbb{R}^n\to \mathbb{C}$ defined by
\begin{equation}
    \phi_X(t)=E\mathrm{exp}(i(t\cdot X)).
\end{equation}
\end{definition}

Next we will recall the concepts of stable and infinitely divisible distributions. 

\begin{definition}
A random vector $X$ has a \textbf{stable distribution} if for independent copies $X_1$ and $X_2$ of $X$, 
and any $a,b>0$, there exist constants $c>0$ and $d\in\mathbb{R}^n$ such that $aX_1+bX_2$ is equal to 
$cX+d$ in distribution.
\end{definition}

\begin{definition}
A random vector $X$ has an \textbf{infinitely divisible} distribution if for all 
positive $m\in\mathbb{N}$, there exist $m$ independent identically distributed random vectors 
$Y_1,...,Y_m$ such that $\sum_{j=1}^nY_j$ has the same distribution as $X$. 
\end{definition}

It is known (as we stated earlier) that all stable distributions are selfdecomposable
(see \cite{sato1999levy}, p. 91), that all selfdecomposable distributions are infinitely divisible 
(see \cite{sato1999levy}, p. 93) and that on $\mathbb{R}$ 
all nondegenerate selfdecomposable distributions are absolutely continuous
(see \cite{sato1999levy}, p. 177) and unimodal (\cite{sato1999levy}, p. 404).

The following is a concept which is used to construct examples of darts with desired properties.

\begin{definition}
Let $X$ and $Y$ be darts taking values in $\mathbb{R}^n$ with laws $\mu_X$ and $\mu_Y$. A dart $Z$ taking values in $\mathbb{R}^n$ is a \textbf{convex combination} of $X$ and $Y$ if its law is given 
by $\mu_Z=p\mu_X+(1-p)\mu_Y$ for some $p\in[0,1]$. If $p\in (0,1)$, then $Z$ is called a \textbf{nontrivial convex combination} of $X$ and $Y$.
\end{definition}

Finally, let us define three function spaces which will be used in some of our theorems. We define $C_c(\mathbb{R}^n)\subseteq C_0(\mathbb{R}^n) \subseteq C_b(\mathbb{R}^n)$ 
by
\begin{equation} 
\label{eq:function.spaces}
\begin{aligned}
C_c(\mathbb{R}^n)&:=\big\{f:\mathbb{R}^n\to\mathbb{R} |\ f \text{ is continuous and has compact support}\}
\\
    C_0(\mathbb{R}^n)&:=\big\{f:\mathbb{R}^n\to\mathbb{R} |\ f \text{ is continuous and } \lim_{||x||\to\infty}f(x)=0\big\}
\\
    C_b(\mathbb{R}^n)&:=\big\{f:\mathbb{R}^n\to\mathbb{R} |\ f \text{ is continuous and bounded} \big\}
\end{aligned}
\end{equation}

The following lemma, whose proof is easy and left to the reader, will be used twice.

\begin{lemma}
\label{lemma:un.cont}
For any finite continuous measure $\mu$ on $\mathbb{R}^n$, one has
$$
\lim_{\delta \to 0} \sup_{x\in  \mathbb{R}} \mu(\{ y : \|x-y\|<\delta\})=0.
$$
\end{lemma}

\subsection{Basic properties of \emph{g}$_{_{\textbf{\emph{X}},\textbf{\emph{f}}}}$}

We next give a few very simple observations concerning our functions $g_{_{X,f}}(d)$.

\begin{proposition}\label{proposition:basic_g_properties}
Let $X$ and $Y$ be two independent darts taking values in $\mathbb{R}^n$, and let $f$ and $h$ be two payoff functions on $\mathbb{R}^n$. If $a_d,a_p,c_p>0$, $d_p \in\mathbb{R}$, and $b_d,b_p\in\mathbb{R}^n$, then the following statements hold
  
\begin{enumerate}
\item  $g_{_{a_dX+b_d,c_pf(a_px+b_p)+d_p}}(d)=c_pg_{_{X,f}}(a_da_pd)+d_p$
\item  $g_{_{X+Y,f}}(d)\leq g_{_{X,f}}(d)$
\item  $g_{_{X,f+h}}(d)\leq g_{_{X,f}}(d)+g_{_{X,h}}(d)$
\item $\inf_x (f(x))\leq g_{_{X,f}}(d) \leq \sup_x (f(x))$
\end{enumerate}
\end{proposition}

\begin{proof}

This proof only requires some simple straightforward computations.

1. We compute
\begin{equation}
    \begin{aligned}
    g_{_{a_dX+b_d,c_pf(a_px+b_p)+d_p}}(d)
        &=
        \sup_a E\left[c_pf\left(a_p(a+d(a_dX+b_d))+b_p\right)+d_p\right]
        \\&=
        c_p(
        \sup_a E\left[f\left((a_pa+a_pdb_d+b_p)+a_da_pdX\right)\right])+d_p
        \\&=
        c_p(\sup_a E\left[f\left(a+a_da_pdX\right)\right])+d_p
        = c_pg_{_{X,f}}(a_da_pd)+d_p.
    \end{aligned}
\end{equation}
2. Due to the independence of $X$ and $Y$ we have that for any $a\in\mathbb{R}^n$
\begin{equation}
\begin{aligned}
E f(a+dX+dY)
&=
\int Ef(a+dy+dX)\mathrm{d}\mu_Y(y)
\\&\leq 
\int \sup_a Ef(a+dX)\mathrm{d}\mu_Y(y)
\\&=
\sup_a Ef(a+dX)=g_{_{X,f}}(d)
\end{aligned}
\end{equation}
and thus
\begin{equation}
g_{_{X+Y,f}}(d)
    =\sup_a \E{f(a+dX+dY)}\leq g_{_{X,f}}(d).
\end{equation}
3 and 4 are easily shown. 
\end{proof}

\subsection{Behavior of reasonableness under projections}
In this subsection, we prove a fairly straightforward result concerning the relationship
between the behavior of a dart in $\mathbb{R}^n$
and the behavior of its various projections  with respect to our questions. 

\begin{proposition}\label{proposition:projections}
Let $X$ be a dart taking values in $\mathbb{R}^n$, and let $h$ be a nonzero linear function 
from $\mathbb{R}^n$ to $\mathbb{R}$.  If
$f$ is a payoff function on $\mathbb{R}$,  then
$g_{_{h(X),f}}(d)=  g_{_{X,f\circ h}}(d)$ so that
$(h(X),f)$ is reasonable if and only if $(X,f\circ h)$ is reasonable.
Hence if $X$ is reasonable, then $h(X)$ is reasonable. 
\end{proposition}

\begin{proof}
We have
$$
   g_{_{h(X),f}}(d)    =     \sup_{a\in\mathbb{R}} Ef(a+dh(X))=
    \sup_{b\in\mathbb{R}^n} Ef(h(b)+dh(X))=
$$
$$
     \sup_{b\in\mathbb{R}^n} Ef(h(b+dX))
    =    \sup_{b\in\mathbb{R}^n} E(f\circ h)(b+dX))
    =    g_{_{X,f\circ h}}(d)
 $$
 where  $h$ being onto was used in the second equality.
The second statement follows immediately.
\end{proof}

\section{Improvement of payoff functions} \label{section:improvement} 
In this section, we will obtain a number of results showing that if a dart is
not reasonable against a certain payoff function, then it will be
not reasonable against a payoff function with perhaps nicer properties. 
Some of these will be used later in the paper.

\subsection{From nonreasonable bounded payoff functions to nonnegative compact support}

\begin{proposition}
\label{proposition:compactPayoff}
Let $X$ be a dart taking values in $\mathbb{R}^n$, and let $f$ be a bounded payoff 
function on $\mathbb{R}^n$. If $(X,f)$ is not reasonable, then there exists a bounded nonnegative payoff function $f'$ with compact support such that $(X,f')$ is not reasonable. If $f$ is continuous, then $f'$ can be taken to be continuous. 
\end{proposition}

\begin{proof}
As $f$ is bounded from below, we may assume that $f\geq 0$.
For any $B>0$ we can define a function $h_B:[0,\infty)\to[0,\infty)$ to be 1 for $x\leq B$, $0$ for 
$x\geq B+1$ and linearly in between, making it continuous.
From this we define a payoff function $f_B$ by
\begin{equation}
    f_B(x):=f(x)\cdot h_B(||x||).
\end{equation}
Note that $f\geq f_B\geq 0$, and that if $f$ is continuous, then so is $f_B$. 
By the monotone convergence theorem we have that 
\begin{equation}
   \lim_{B\to\infty} Ef_B(a+dX)= Ef(a+dX),\ \forall a\in\mathbb{R}^n,\ d>0
\end{equation}
which easily yields
\begin{equation}
\lim_{B\to\infty} \g{X,f_B}(d)\geq \g{X,f}(d)
\end{equation}
for all $d$. Since the reverse inequality is trivial, we obtain
\begin{equation}
    \lim_{B\to\infty} \g{X,f_B}(d)= \g{X,f}(d),\ \forall d>0.
\end{equation}
Some $(X,f_B)$ must not be reasonable since otherwise
$\g{X,f_B}$ would be decreasing in $d$ for all $B$ implying that
$\g{X,f}$ is decreasing, a contradiction.
\end{proof}

A variant of the proof of Proposition \ref{proposition:compactPayoff}
explains why we have defined our payoff functions to be bounded from above.

\begin{proposition}
\label{proposition:bded.from.above}
Let $X$ be a dart taking values in $\mathbb{R}^n$. Let $f:\mathbb{R}^n\to \mathbb{R}$. 
(Note $f$ is not assumed to be bounded above and so it is not necessarily a payoff function.)
Assume that $Ef(a+dX)$ is well defined and finite for all  $a\in\mathbb{R}^n,\ d>0$.
Letting, as we do for payoff functions,
\begin{equation}
g_{_{X,f}}(d):=\sup_{a\in\mathbb{R}^n} Ef(a+dX),
\end{equation}
assume that there exist $d_1< d_2$ so that $g_{_{X,f}}(d_1) < g_{_{X,f}}(d_2)$. 
Then there exists a payoff function $f'$ (hence bounded from above by definition)
so that $(X,f')$ is not reasonable.  If $f$ is continuous, then $f'$ may be taken to be continuous. 
\end{proposition}

\begin{proof} 
For $M>0$, let $ f_M(x)=\max\{f(x), M\}$ which is continuous if
$f$ is.
By the Lebesgue dominated convergence theorem,
\begin{equation}
   \lim_{M\to\infty} Ef_M(a+dX)= Ef(a+dX),\ \forall a\in\mathbb{R}^n,\ d>0
\end{equation}
easily leading to
\begin{equation}
    \lim_{M\to\infty} \g{X,f_M}(d)= \g{X,f}(d),\ \forall d>0.
\end{equation}
Exactly as in the proof of \ref{proposition:compactPayoff},  we have that
for some $M$, $(X,f_M)$ is not reasonable.
\end{proof}

\subsection{Absolutely continuous darts: making noncontinuous payoff functions continuous}

\begin{proof}[Proof of Theorem \ref{theorem:absCont}]
By Proposition \ref{proposition:compactPayoff}, we can replace $f$ by a nonnegative bounded
function with compact support which is nonreasonable. 
Letting $m$ denote Lebesgue measure, by Lusin's Theorem (see \cite{folland1999real}, p. 217), for any $\epsilon>0$ there exists a measurable set $A\subseteq\mathbb{R}^n$ and a continuous nonnegative function 
$h_\epsilon\in C_c(\mathbb{R}^n)$, such that $f=h_\epsilon$ on 
$A$, $m(A^c)<\epsilon$, and $\sup_x |h_\epsilon (x)|\leq \sup_x |f(x)|$.

For any $d\geq 0$ we have that 
\begin{equation}
\begin{aligned}
    |Ef(a+dX)-Eh_\epsilon (a+dX)|
    &\leq\int_{\{x:a+dx\in A^c\}}|f(a+dx)-h_\epsilon(a+dx)|\mathrm{d}\mu_X(x)
    \\&\leq 
    2\sup_y |f(y)| \mu_X(\{x:a+dx\in A^c\})
\end{aligned}
\end{equation}
and thus 
\begin{equation}
    \begin{aligned}
        \g{X,h_\epsilon}(d)
        &\leq
        \sup_a \left(Ef(a+dX)+2\sup_y |f(y)| \mu_X(\{x:a+dx\in A^c\})\right)
        \\
        \g{X,h_\epsilon}(d)&\geq 
        \sup_a \left(Ef(a+dX)-2\sup_y |f(y)| \mu_X(\{x:a+dx\in A^c\})\right)
\end{aligned}
\end{equation}
If $(X,f)$ is not reasonable, then there exists $d_1,d_2\geq 0$ such that $d_1<d_2$ and $\g{X,f}(d_1)<\g{X,f}(d_2)$. As $\mu_X$ is an absolutely continuous finite measure there exists a $\delta>0$ such that for any measurable set $E$ with $m(E)<\delta$, we have that
\begin{equation}
\mu_X(E)<\frac{\g{X,f}(d_2)-\g{X,f}(d_1)}{8\sup_y |f(y)|}.    
\end{equation}
Now note that by the properties of Lebesgue measure, $m(\{x:a+dx\in A^c\})=m((A^c-a)/d)=m(A^c)/d^n<\epsilon/d^n$, and now choose $\epsilon>0$ so that $\epsilon/d_1^n<\delta$. For all
$d\geq d_1$ we now get
\begin{equation}
    \begin{aligned}
        \g{X,h_\epsilon}(d)
        &\leq
        \sup_a \left(Ef(a+dX)
        +2\sup_y |f(y)|\frac{\g{X,f}(d_2)-\g{X,f}(d_1)}{8\sup_y |f(y)|}\right)
        \\&=
        g_{_{X,f}}(d)+\frac{\g{X,f}(d_2)-\g{X,f}(d_1)}{4}
    \end{aligned}
\end{equation}
and
\begin{equation}
    \begin{aligned}
       \g{X,h_\epsilon}(d)&\geq 
        \sup_a \left(Ef(a+dX)-2\sup_y |f(y)| \frac{\g{X,f}(d_2)-\g{X,f}(d_1)}{8\sup_y |f(y)|}\right)
        \\&=
        g_{_{X,f}}(d)-\frac{\g{X,f}(d_2)-\g{X,f}(d_1)}{4}
    \end{aligned}
\end{equation}
 We now get that 
\begin{equation}
\begin{aligned}
    \g{X,h_\epsilon}(d_2)-\g{X,h_\epsilon}(d_1)
    &\geq
    \g{X,f}(d_2)-\g{X,f}(d_1)
    -\frac{\g{X,f}(d_2)-\g{X,f}(d_1)}{2}
    \\&=\frac{\g{X,f}(d_2)-\g{X,f}(d_1)}{2}>0
\end{aligned}
\end{equation}
and thus $(X,h_\epsilon)$ is not reasonable. 
\end{proof}

\subsection{Making payoff functions with a countable number of discontinuities continuous}

\begin{theorem}
\label{theorem:countablePayoff}
Assume that $X$ is a continuous dart taking values in $\mathbb{R}^n$ and $f$ is a bounded
payoff function on $\mathbb{R}^n$ such that $(X,f)$ is not reasonable. Then if $f$ has at most a countable number of discontinuities, 
then there exists a continuous nonnegative payoff function $h$ with compact support on $\mathbb{R}^n$ such that $(X,h)$ is not reasonable.
\end{theorem}

\begin{proof}
In view of Proposition \ref{proposition:compactPayoff}, it suffices to find an $h$ which is bounded and continuous which we now do.

If $(X,f)$ is not reasonable, there exist $d_1,d_2\geq 0$ such that $d_1<d_2$ and $\g{X,f}(d_1)<\g{X,f}(d_2)$. 
Let $\{x_k\}_{k=1}^\infty$ be the discontinuity points of $f$.

By Lemma \ref{lemma:un.cont}, we have that
for any $\epsilon>0$ there exists a sequence of positive numbers $\{\delta_k\}_{k=1}^\infty$ such that for all $k$ we have 
\begin{equation}
    \mu_X\left(\large\{y\in\mathbb{R}^n : ||x-y||<\frac{\delta_k}{d_1}\large\}\right)<\frac{\epsilon}{2^k},\ \forall x\in\mathbb{R}^n.
    \label{eq:choice_of_delta_n}
\end{equation}
Now define the set, which depends upon $\epsilon$
\begin{equation}
    A=\bigcup_{k=1}^\infty \{z\in\mathbb{R}^n : ||x_k-z||<\delta_k\}.
\end{equation}
Note that $A$ is an open set which contains all of the discontinuities of $f$. Thus $f$ is continuous on the closed set $A^c$, 
and therefore by the Tietze Extension Theorem, there exists a continuous function $h_\epsilon$ on $\mathbb{R}^n$ such that $h_\epsilon$ is equal to $f$ on $A^c$ and $\sup_x |h_\epsilon(x)|\leq \sup_x |f(x)|$.
We now have
\begin{equation*}
\begin{aligned}
    |Ef(a+dX)-Eh_\epsilon (a+dX)|&\le \int_{\{x:a+dx\in A\}}|f(a+dx)-h_\epsilon(a+dx)|\mathrm{d}\mu_X(x)
    \\&\leq 
    2\sup_y |f(y)| \mu_X(\{x:a+dx\in A\})
    \\&\leq
    2\sup_y |f(y)| \sum_{k=1}^\infty
    \mu_X(\left\{x:a+dx\in  \{z\in\mathbb{R}^n : ||x_k-z||<\delta_k\}\right\})
    \\&\leq 
    2\sup_y |f(y)| \sum_{k=1}^\infty
    \mu_X\left(
            \{x\in\mathbb{R}^n : ||\frac{x_k-a}{d}-x||<\frac{\delta_k}{d}\}
        \right).
\end{aligned}
\end{equation*}
Equation \eqref{eq:choice_of_delta_n} now gives us that for all $d\geq d_1$
\begin{equation}
    |Ef(a+dX)-Eh_\epsilon (a+dX)|\leq 
    2\sup_y |f(y)|\epsilon.
\end{equation}
From this we get that for all $d\geq d_1$
\begin{equation}
    \begin{aligned}
        \g{X,h_\epsilon}(d)&\leq g_{_{X,f}}(d)+2\sup_y |f(y)|\epsilon
        \\
        \g{X,h_\epsilon}(d)&\geq g_{_{X,f}}(d)-2\sup_y |f(y)|\epsilon
    \end{aligned}
\end{equation}
which implies
\begin{equation}
    \begin{aligned}
        \g{X,h_\epsilon}(d_2)-\g{X,h_\epsilon}(d_1)
        &\geq
        \g{X,f}(d_2)-\g{X,f}(d_1)
        -
        4\sup_y |f(y)|\epsilon.
    \end{aligned}
\end{equation}
Thus if we choose 
\begin{equation}
    \epsilon<\frac{\g{X,f}(d_2)-\g{X,f}(d_1)}{4\sup_y |f(y)|}
\end{equation}
then we see that $(X,h_\epsilon)$ is not reasonable. 
\end{proof}


\section{Selfdecomposable distributions are reasonable}  \label{section:selfdecomposible} 

We begin this section with proving Theorem \ref{theorem:reasonable.dart}.

\begin{proof}[Proof of Theorem \ref{theorem:reasonable.dart}]
Fix $f$ and $s>0$. From $X|dX$ it follows
that $sX|dsX$. Choose a random variable $Z$ so that if $Z$ and $X$ are independent,
then $sX+Z$ and $dsX$ have the same distribution. By Proposition \ref{proposition:basic_g_properties} we have that
\begin{equation}
    g_{_{X,f}}(s)= g_{_{sX,f}}(1) \geq g_{_{sX+Z,f}}(1) =
    g_{_{dsX,f}}(1)=g_{_{X,f}}(ds).
\end{equation}
\end{proof}

\begin{remark}
The point of $X|dX$ is that you can then simulate being at distance $ds$
when you are standing at distance $s$ by randomizing your target. Hence you can do at least
as well at distance $s$ as at distance $ds$ with respect to any payoff function.
\end{remark}

As all selfdecomposable distributions are infinitely divisible (see \cite{sato1999levy}, p. 93), it is natural to ask whether all darts that have infinitely divisible distributions are reasonable. 
However, it is immediate from either Theorem \ref{theorem:cfeasy}
or Theorem \ref{theorem:2pointmasses} that the Poisson distribution (the building block of
almost all infinitely divisible distributions) is not reasonable.

We end this section by listing some examples of selfdecomposable distributions which are not stable.
We won't list the original papers where these were proved but by looking at 
\cite{sato1999levy} and \cite{SH} one obtains almost all of these as well as others.

The list is as follows: all gamma distributions, Laplace distribution, Pareto distribution, Gumbel distribution,
logistic distribution, log-normal distribution, F-distribution, t-distribution,
hyperbolic-sine and hyperbolic-cosine distributions,
the Beta distribution of the second kind (sometimes called the beta prime distribution,
not to be confused with the ordinary Beta distribution), so-called "generalized Gamma convolutions" 
(see \cite{Bond})  and the half Cauchy distribution (which is interesting in light of the known fact that
the half normal distribution is not even infinitely divisible and hence not
selfdecomposable, as well as not being reasonable as we saw in the introduction). 

\section{Reasonableness with respect to $\cos(x)$}  \label{section:cosine}

In this section, we will prove Theorem \ref{theorem:cf} (which is 
a strengthening of Theorem \ref{theorem:cfeasy}), Theorem \ref{theorem:2pointmasses},
Proposition \ref{proposition:cfnotToZero}  and Theorem \ref{theorem:reasPlusNonreasCosine}.

\subsection{The Cosine payoff function}

We now state and prove an extension of Theorem \ref{theorem:cfeasy}. It will be an extension since it will
deal with darts in  $\mathbb{R}^n$ and because we will obtain an explicit formula 
for $g_{_{X,f}}(d)$ which will lead to a necessary and sufficient condition for reasonableness
with respect to our "cosine function".

\begin{theorem}\label{theorem:cf}
Let $X$ be any dart taking values in $\mathbb{R}^n$ with characteristic function $\phi_X$, and 
let $f(x):=\cos \left(\sum_{j=1}^n x_j\right)$. Then for any $d>0$ we have that
\begin{equation}
    Ef(a+dX)
=
|\phi_X(d\vec{1})|\cos\left(\sum_{j=1}^n a_j+\text{Arg}(\phi_X(d\vec{1}))\right),
\label{eq:expectedCosine}
\end{equation}
where $\vec{1}=(1,1,...,1)$. 

In particular this implies that if $f(x)=\cos \left(\sum_{j=1}^n x_j\right)$, 
then $g_{_{X,f}}(d)=\left|\phi_X(d\vec{1})\right|$, 
and hence  $(X,f)$ is reasonable if and only if $\left|\phi_X(d\vec{1})\right|$ is 
decreasing in $d$ on $(0,\infty)$.
\end{theorem}

\begin{proof}
We have that 
\begin{equation}
\begin{aligned}
    Ef(a+dX) &= E \cos \left(\sum_{j=1}^n a_j+d\Vec{1}\cdot X \right)
    =
    \text{Re}\left(
                E \text{exp} \left(i\sum_{j=1}^n a_j+id\Vec{1}\cdot X \right)
            \right)
    \\&=
    \text{Re}\left(
                e^{i\sum_{j=1}^n a_j}
                E 
                e^{id\Vec{1}\cdot X}
            \right)
    =
    \text{Re}\left(
                e^{i\sum_{j=1}^n a_j}
                \phi_X(d\vec{1})
            \right)
    \\&=
    \text{Re}\left(
                e^{i\sum_{j=1}^n a_j}
                |\phi_X(d\vec{1})|
            e^{i\text{Arg}(\phi_X(d\vec{1}))}
            \right)
    \\&=
    |\phi_X(d\vec{1})|\cos\left(\sum_{j=1}^n a_j+\text{Arg}(\phi_X(d\vec{1}))\right)
\end{aligned}
\end{equation}
and thus
\begin{equation}
    g_{_{X,f}}(d)=|\phi_X(d\vec{1})|.
\end{equation}
\end{proof}

Besides giving us the  means to investigate reasonableness, Theorem \ref{theorem:cf} also implicitly 
tells us in one dimension which $a$'s 
maximize $E\cos(a+dX)$. Note that the points of maximization can move around in
a  discontinuous way, as the following example demonstrates.

\begin{example}
Let $X$ be a dart taking values in $\mathbb{R}$ such that $P(X=1)=P(X=-1)=1/2$. The characteristic function of $X$ is $\phi_X(t)=\cos(t)$, and by Theorem \ref{theorem:cf} we see that $X$ is not reasonable. Furthermore, by the same theorem, $E(\cos(a+dX))$ is always optimized at $a=-$Arg$(\phi_X(d))+2k\pi$, $k\in\mathbb{Z}$, and so as $d$ changes, the optimal place to aim switches back and forth between $2k\pi$ and $\pi+2k\pi$. 
\end{example}

\begin{remark}
(i) Note that when $d\geq 2$ the set of $a$'s which maximizes \eqref{eq:expectedCosine} is very large. 
For any fixed $a_1,...,a_{n-1},$ there are infinitely many $a_n$ such that $a=(a_1,...,a_n)$ 
maximizes $Ef(a+dX)$.\newline
(ii) Theorem \ref{theorem:cf} implies that if $\phi_X(d\vec{1})=0$, 
then it does not matter where we aim when we are at distance $d$. 
\end{remark}

There are two simple ways of combining random variables, adding independent copies 
or taking convex combinations. Using Theorem \ref{theorem:cf} and 
elementary properties of characteristic functions (including the fact
that the characteristic function of a symmetric random vector is real-valued) one
easily obtains the following two corollaries.

\begin{corollary}\label{corollary:independentDartsCosine}
  Let $X$ and $Y$ be two independent darts taking values in $\mathbb{R}^n$ and let 
$f(x):=\cos \left(\sum_{j=1}^n x_j\right)$.
  If $(X,f)$ and $(Y,f)$ are both reasonable, then $(X+Y,f)$ is also reasonable.
\end{corollary}

\begin{corollary}
\label{cor:symmetricConvex}
Let $X$ and $Y$ be two independent darts taking values in $\mathbb{R}^n$ 
which are symmetric about the origin and let 
$f(x):=\cos \left(\sum_{j=1}^n x_j\right)$.
If $(X,f)$ and $(Y,f)$ are both reasonable, then $(Z,f)$ is also reasonable,
where $Z$ is any convex combination of $X$ and $Y$.
\end{corollary}

\begin{remark}
A special case would be if $X$ is a point mass of weight 1 at 0 and $Y$ is a 
standard normal distribution. Then any convex combination of them would be
reasonable with respect to $\cos(x)$. However, one can check (by computing the
characteristic function) that if we modify $Y$ by adding a constant but leave $X$ as is,
then any nontrivial convex combination will not be reasonable with respect to $\cos(x)$. 
\end{remark}

\subsection{Two point masses and almost periodicity}

We now move on to the proof of Theorem \ref{theorem:2pointmasses} and 
Proposition \ref{proposition:cfnotToZero}.

\begin{proof}[Proof of Theorem \ref{theorem:2pointmasses}]
We first prove the first statement concerning random variables.

We begin by giving the idea of the proof. The characteristic function $\phi_d(t)$ for
the (normalized) discrete component of the distribution is what is called almost periodic. This means it
is periodic up to a small error. Since there is at least two point masses, 
$|\phi_d(t)|$ becomes less than 1. Therefore, by almost periodicity,
$|\phi_d(t)|$ is not monotone. We now need to make sure, in view of Theorem \ref{theorem:cfeasy},
that the contribution of the continuous part 
of the characteristic function, $|\phi_c(t)|$, does not destroy this nonmonotonicity.
However, it is known that $|\phi_c(t)|$ goes to 0 in a Cesaro sense. 
Finally, the almost periodicity gives us that the nonmonotonicity of $|\phi_d(t)|$
occurs on a periodic basis and hence can be shown to occur when $|\phi_c(t)|$ is small, 
thereby not destroying the nonmonotonicity.

We now begin the proof.
Partition the distribution of $X$, $\mu_X$, into its atomic and continuous pieces
$$
\mu_X=p\mu_d+(1-p) \mu_c
$$
where $\mu_d$ and  $\mu_c$ are then probability measures. We then have
$$
\phi_X(t)= p \phi_d(t)+(1-p) \phi_c(t)
$$
where $\phi_d$ and  $\phi_c$ are the characteristic functions corresponding to  $\mu_d$ and  
$\mu_c$.  

First, we know from \cite{chung2001course} (Theorem 6.2.5, p. 164)  that 
\begin{equation}
\label{eq:chung}
 \lim_{T\to\infty} \frac{1}{T}\int_{0}^{T}|\phi_c(t)|^2\mathrm{d}t = 0.
\end{equation}
Since $\phi_c(t)$ is also uniformly continuous, it follows that for every $\eps>0$, $L$ and $M$,
there exists $x=x(\eps,L,M)\ge M$ so that 
$$
|\phi_c(t)|\le \eps \mbox{ on } [x,x+L].
$$

Next, since there are at least two point masses,  $\mu_d$ is nondegenerate and hence
 $|\phi_d(t)|$ is not constant. In particular, there exists $t_0>0$ and $\eps_0>0$ so that
 $$
|\phi_d(t_0)|< 1-\eps_0.
$$
Next it is known (see \cite{BS}, p.43-44) that $\phi_d(t)$ is an almost periodic function.
This implies (see \cite{BS}, p.43-44) that there exists $L$ so that every interval of length 
$L$ in $\mathbb{R}$ contains a $\tau$ so that
\begin{equation}
\label{eq:Linfinity}
\|\phi_d(t+\tau)-\phi_d(t)\|_\infty <  \frac{\eps_0}{10}.
\end{equation}

With this $L$, by the above, choose $x=x(\frac{p\eps_0}{10},L,t_0)\ge t_0$ so that 
\begin{equation}
\label{eq:cont.good}
|\phi_c(t)|\le \frac{p\eps_ 0}{10}\mbox{ on } [x,x+L].
\end{equation}
Choosing now $\tau\ge 0$ as above in the interval $[x-t_0,x-t_0+L]$,
we have by \eqref{eq:Linfinity} that 
$$
|\phi_d(t_0+\tau)| <  1- \frac{3\eps_0}{4}. 
$$
This implies, using \eqref{eq:cont.good}, that 
$$
|\phi_X(t_0+\tau)| <  p(1- \frac{3\eps_0}{4}) +\frac{(1-p)p\eps_ 0}{10}\le 
p(1- \frac{\eps_0}{2}).
$$

Next, choose $y=y(\frac{p\eps_0}{10},L,t_0+\tau)\ge t_0+\tau$ so that 
\begin{equation}
\label{eq:cont.goodagain}
|\phi_c(t)|\le \frac{p\eps_ 0}{10}\mbox{ on } [y,y+L].
\end{equation}
Choosing now $\tau'$ as above in the interval $[y,y+L]$,
we have by \eqref{eq:Linfinity} that 
$$
|\phi_d(\tau')| >  1- \frac{\eps_0}{10}. 
$$
This implies, using \eqref{eq:cont.goodagain} that 
$$
|\phi_X(\tau')| >  p(1- \frac{\eps_0}{10}) -\frac{(1-p)p\eps_ 0}{10}\ge 
p(1- \frac{\eps_0}{5}).
$$

We therefore have that $\tau'\ge t_0+\tau>0$ but
$$
|\phi_X(\tau')| \ge p(1- \frac{\eps_0}{5}) > p(1- \frac{\eps_0}{2}) \ge |\phi_X(t_0+\tau)| .
$$
This implies that $|\phi_X(t)|$ is not decreasing in $t$ and hence by 
Theorem \ref{theorem:cf}, $(X,\cos(x))$ is not reasonable.

The random vector case easily 
follows from the 1-dimensional case just proved, the fact that if $X$ has at least two
point masses, then at least one of the marginals has two point masses and
Proposition \ref{proposition:projections}.
\end{proof}

\begin{proof}[Proof of Proposition \ref{proposition:cfnotToZero}]
Assume that $X$ is reasonable with respect to $\cos(x)$. Then, by Theorem \ref{theorem:cfeasy},
$|\phi_X(t)|$ must be decreasing in $t$ on $[0,\infty)$. 
However, by \cite{chung2001course} (Theorem 6.2.5, p. 164) we have, since the distribution
is continuous, that 
\begin{equation}
 \lim_{T\to\infty} \frac{1}{T}\int_{0}^{T}|\phi_X(t)|^2\mathrm{d}t =0.
 \end{equation}
Together these imply that $\phi_X(t)$ goes to zero as $t\to\infty$,
which gives a contradiction. Thus $X$ is not reasonable with respect 
to $\cos(x)$. 
\end{proof}

\subsection{An example with a phase transition}

We now move on to the proof of Theorem \ref{theorem:reasPlusNonreasCosine}.

\begin{proof}[Proof of Theorem \ref{theorem:reasPlusNonreasCosine}]
By Theorem \ref{theorem:cf} $X$ is reasonable with respect to $\cos(x)$ if and only if $|\phi_X(d)|$ is decreasing in $d$, $d>0$. Due to independence, the characteristic function of $X$ is
\begin{equation}
    \phi_X(d)=\phi_{X_1}(d)\phi_{X_2}(d)
    =
    (1-p+p e^{id})\text{exp}(-\sigma^2d^2/2).
\end{equation}
The absolute value of this is decreasing if and only if
\begin{equation}
    |\phi_X(d)|^2
    =
    \left(p^2+(1-p)^2+2(1-p)p\cos(d)\right)\text{exp}(-\sigma^2d^2)
\end{equation}
is decreasing. This in turn is decreasing if and only if its derivative with respect to $d$ is nonpositive 
on $[0,\infty)$. We have that 
\begin{equation}
    \frac{\text{d}}{\text{d}d}\left(|\phi_X(d)|^2\right)
    =
    -2
    \text{exp}(-\sigma^2d^2)\Big[
    \sigma^2d\Big(p^2+(1-p)^2+2(1-p)p\cos(d)\Big)
    +
    (1-p)p\sin(d)
    \Big]
\end{equation}
and thus $X$ is reasonable with respect to $\cos(x)$ if and only if
\begin{equation}
    \sigma^2d\Big(p^2+(1-p)^2+2(1-p)p\cos(d)\Big)+(1-p)p\sin(d)\geq 0,\ \forall d\geq 0.
\end{equation}
Note that if $p=1/2$, then it is easy to see that $|\phi_X(\pi+2m\pi)|=0$ for all $m\in\mathbb{N}$, but the characteristic function is still not identically zero, and is thus not decreasing. Now assume that $p\neq 1/2$. We have that 
\begin{equation}
    p^2+(1-p)^2+2(1-p)p\cos(d)
    =
    |1-p+pe^{id}|^2
    \geq 
    |1-2p|^2>0,\ \forall d.
\end{equation}
Thus 
\begin{equation}
    \sigma^2d\Big(p^2+(1-p)^2+2(1-p)p\cos(d)\Big)+(1-p)p\sin(d)
    \geq 
    \sigma^2d|1-2p|^2
    +
    (1-p)p\sin(d).
\end{equation}
As $\sin (d)\geq 0$ for $d\in [0,\pi]$, it is easy to see that if $\sigma^2\pi|1-2p|^2\geq(1-p)p$ then 
\begin{equation}
    \sigma^2d\Big(p^2+(1-p)^2+2(1-p)p\cos(d)\Big)+(1-p)p\sin(d)\geq 0,\ \forall d\geq 0.
\end{equation}
Thus $p\neq 1/2$ and $\sigma^2\geq (1-p)p/(\pi|1-2p|^2)$ is a sufficient condition for $X$ to be reasonable with respect to $\cos(x)$.

Since, for all $p\in (0,1)$,  the Fourier transform of Bern($p$) has a zero in the complex plane,
the final claim follows from Theorem  \ref{theorem:entire}.
\end{proof}

\section{Compactly supported darts}  \label{section:compactsupport} 
In this section, we prove Theorems \ref{theorem:compactDart} and \ref{theorem:entire}.

\begin{proof}[Proof of Theorem \ref{theorem:compactDart}]
In view of Proposition \ref{proposition:compactPayoff}, it is sufficient to find a bounded continuous payoff 
function with respect to which $X$ is not reasonable. Next, by Proposition
\ref{proposition:projections}, it then suffices to show that this holds in one dimension.

By \cite{lukacs1970characteristic} (see Theorem 7.2.3, p. 202) we have that the characteristic function of $X$, $\phi_X$, is an entire function with infinitely many zeros, none of
which of course lie on the imaginary axis. 
Let $z_0$ be any zero of $\phi_X$. As the characteristic function is entire, all of its zeros are isolated, and thus there exists a $d_0>1$ such that $\phi_X(d_0z_0)\neq 0$.

Now let $c,\omega \in \mathbb{R}$ be defined so that 
\begin{equation}
    z_0=\omega-ic, \,\, \omega\neq 0
\end{equation}
and define the function $f:\mathbb{R}\to \mathbb{R}$ by 
\begin{equation}
   f(x):=e^{cx}\cos(\omega x)=\text{Re}\left(e^{(c+i \omega)x}\right)=\text{Re}\left(e^{iz_0x}\right). 
\end{equation}
Note that $f$ is not bounded. For any $d>0$ and $a\in\mathbb{R}$ we have that
\begin{equation}
    Ef(a+dX)=e^{ca}\text{Re}\left(e^{i\omega a}Ee^{iz_0dX}\right)=e^{ca}\text{Re}\left(e^{i\omega a}\phi_X(dz_0)\right)
    \label{eq:unboundedExpectation}
\end{equation}
and so by taking $d=1$ we get
\begin{equation}
    Ef(a+X)=0,\ \forall a.
\end{equation}
Furthermore, for $a_0:=-\text{Arg}\left(\phi_X(d_0z_0)\right)/\omega$ this gives us
\begin{equation}
    Ef\left(a_0+d_0X\right)
    =
    \text{exp}\left(-\frac{c\text{Arg}\left(\phi_X(d_0z_0)\right)}{\omega}\right)
    |\phi_X(d_0z_0)|>0
\end{equation}
Since $d_0 >1$, this gives us the type
of nonreasonable behavior we are after. We now however
have to modify $f$ so that it is bounded while maintaining this behavior.

As $X$ is bounded, there is a $B>0$ such that $P(d_0|X|\leq B)=1$. Now let us define the payoff function $h$ by $h(x):=f(x)$ for $|x|\leq |a_0|+B$, $h(x):=-\sup_{|y|\leq |a_0|+10B}(|f(y)|)$ for $|a_0|+2B\leq |x|$, and for $|a_0|+B\leq |x|\leq |a_0|+2B$ it is defined as
\begin{equation}
    h(x):=
     \frac{|x|-(|a_0|+B)}{B}\left(-f(x)-\sup_{|y|\leq |a_0|+10B}(|f(y)|)\right)+f(x).
\end{equation}
Note that $h$ is continuous and bounded and that 
\begin{equation}
\begin{aligned}
    h(x)&\leq f(x),\ |x|\leq |a_0|+10B\\
    h(x)&\leq 0,\ |a_0|+2B\leq |x|.
\end{aligned}
\end{equation}
To see the first of these inequalities, note that for $|a_0|+B\leq |x|\leq |a_0|+2B$, $h(x)$ is equal to $f(x)$ plus a nonpositive term. 

With this definition we have that  
\begin{equation}
    Eh(a_0+d_0X)=Ef(a_0+d_0X)>0
\end{equation}
and 
\begin{equation}
\begin{aligned}
    Eh(a+X)&\leq Ef(a+X)=0,\ 0\leq |a| \leq |a_0|+9B\\
    Eh(a+X)&\leq 0,\ |a_0|+9B\leq |a|.
\end{aligned}
\end{equation}
Thus $\g{X,h}(1)\leq 0$ but $\g{X,h}(d_0)> 0$ implying 
that $(X,h)$ is not reasonable.  
\end{proof}

\begin{remark}
1.  Note that in \eqref{eq:unboundedExpectation}, if $c$ and $\phi_X(dz_0)$ are nonzero, then one can make $Ef(a+dX)$ arbitrarily large by choosing $a$ appropriately. Thus if we allowed for unbounded payoff functions, we would have that $\g{X,f}(d)\in\{0,\infty\}$ for all $d>0$. \\
2. The second half of the proof applies more generally and shows that if one has a compact dart
and a continuous function which is "nonreasonable", then $f$ can be modfied to be bounded and continuous.
\end{remark}

\begin{proof}[Proof of Theorem \ref{theorem:entire}]
The same proof works for both statements.
By the proof of  Theorem \ref{theorem:compactDart}, we have that if 
$z_0=\omega-ic, \,\, \omega\neq 0$ is a zero of the Fourier transform within its strip of analyticity
and 
\begin{equation}
   f(x):=e^{cx}\cos(\omega x)=\text{Re}\left(e^{(c+i \omega)x}\right)=\text{Re}\left(e^{iz_0x}\right),
\end{equation}
then 
\begin{equation}
\sup_{a\in\mathbb{R}} Ef(a+X) =0
\end{equation}
and for $d$ slightly larger than 1,
\begin{equation}
\sup_{a\in\mathbb{R}} Ef(a+dX) >0.
\end{equation}
We can now apply Proposition \ref{proposition:bded.from.above} to conclude that $X$ is not
reasonable against some continuous payoff function.
\end{proof}

\section{Having a point mass implies not reasonable}  \label{section:pointmass} 

In this section, we prove Theorem \ref{theorem:aNonTrivialPointmass}.

\begin{proof}[Proof of Theorem \ref{theorem:aNonTrivialPointmass}]
In view of Proposition \ref{proposition:compactPayoff}, it is sufficient to find a bounded continuous payoff 
function with respect to which $X$ is not reasonable.

We assume without loss of generality that $P(X=0)>0$. 
Since $X=(X_1,...,X_n)$ is nondegenerate, at least one of 
$X_1,...,X_n$ must be nondegenerate which we assume to be $X_1$.
Let $h$ be defined by $h(x)=x_1$. If we find a payoff function $k$ for $X_1$
so that $(X_1,k)$ is not reasonable, then by Proposition \ref{proposition:projections}, 
we will have that $(X, k \circ h)$ is not reasonable.  Also, $k$  being bounded and continuous 
implies that $k \circ h$ is also. Hence it suffices to consider the 1-dimensional case. 

If $X_1$ has two or more point masses, then Theorem \ref{theorem:2pointmasses}
yields our function $k$. Otherwise, we may assume that $X_1$ has only 0 as a point mass.
In view of Proposition  \ref{proposition:compactPayoff}, we need only find a continuous 
bounded payoff function with respect to which $X_1$ is not reasonable.

In this case, for all $\delta\in (0,1/2)$ we define
\begin{equation}
    k_\delta(x)=
    \begin{cases}
    1-\frac{|x|}{\delta},\ |x|\leq \delta\\
    0,\ \delta<|x|<1-\delta\\
    \frac{P(X_1=0)}{2}\cdot (\frac{|x|}{\delta}+1-\frac{1}{\delta}), 1-\delta\leq |x|\leq 1\\ 
     \frac{P(X_1=0)}{2},\ |x|\geq 1.
    \end{cases}
\end{equation}
Note that $k_\delta$ is continuous and bounded by 1 in absolute value for all $\delta$.
It is easy to show that for any fixed $\delta$
\begin{equation}
    \lim_{d\to\infty} \g{X,k_\delta}(d)\geq P(X_1=0)+ \frac{P(X_1=0)}{2},
\end{equation}
and so we only need to show that there exist $d>0$ and $\delta$
such that $\g{X_1,k_\delta}(d)<P(X_1=0)+ \frac{P(X_1=0)}{2}$. We will do this by showing that we can find  $d>0$ and $\delta$ 
so that $\g{X_1,k_\delta}(d)$ is arbitrarily close to $P(X_1=0)$. 

For any $\epsilon>0$ we can choose a $d_0>0$ such that 
\begin{equation}
    P\left(|d_0X|>\frac{1}{2}\right)<\epsilon.
    \end{equation}
    
Next, using Lemma \ref{lemma:un.cont}, it can easily be shown that 
\begin{equation}
\label{eq:deltaLim}
       \lim_{\delta\to 0} \left(\g{X_1,k_\delta}(d_0)\right) \leq   P(X_1=0)    +   \frac{P(X_1=0)}{2} \epsilon.
       \end{equation}
Thus there exist $d>0,\delta\in (0,1/2)$ 
such that 
$$
\g{X_1,k_\delta}(d)<P(X_1=0)+ \frac{P(X_1=0)}{2}
$$
as desired.
\end{proof}

\medskip
The following example demonstrates other interesting things which can occur with a single point mass. Namely, there
is a pair $(X,f)$ with $f$ bounded such that $\g{X,f}(d)$ is strictly increasing in $d$.
Let $X$ have law
\begin{equation*}
    \mu_X=\frac{\delta_0}{2}+\frac{\mu_Z}{2}
\end{equation*}
where  $Z$ is $N(0,1)$.
Let $f$ be a payoff function on $\mathbb{R}$ defined by
\begin{equation*}
    f(x)
    =
    \begin{cases}
    1,\ x=0\\
    0,\ 0<|x|<1\\
    \frac{1}{2}, |x|>1.
    \end{cases}
\end{equation*}
It is elementary and left to the reader to check that 
\begin{equation}
    \g{X,f}(d)   =    \frac{1}{2}+\frac{1}{4} P(|Z|>\frac{1}{d}).
\end{equation}
Note that it is impossible to construct an example of a strictly increasing $\g{X,f}(d)$ where $f$ is bounded and continuous, as it can be shown that for any dart $X$ and any bounded continuous payoff function $f$

\begin{equation*}
    \lim_{d\to 0}\g{X,f}(d) = \sup_x f(x)= \sup_{d>0} \g{X,f}(d) .
\end{equation*}

\section{Singular measures and reasonableness} \label{section:sing} 
In this section, we prove Theorem \ref{theorem:singular}.
We mention that the proof is similar to the proof of Theorem \ref{theorem:aNonTrivialPointmass}.

\begin{proof}[Proof of Theorem \ref{theorem:singular}]
In view of Proposition \ref{proposition:compactPayoff}, it is sufficient to find a bounded
payoff function with respect to which $X$ is not reasonable.

Write the distribution of $X$ as $p\mu_s+(1-p)\mu_{ac}$ where $p >0$, $\mu_s$ is a singular probability
measure, $\mu_{ac}$ is an absolutely continuous probability measure 
and $N$ is a Lebesgue null set on which  $\mu_s$ is concentrated.
Without loss of generality $N\subseteq [-1,1]$. We can assume that $p<1$ and 
$\mu_{ac}$ is not compactly supported since otherwise the result would follow from 
Theorem \ref{theorem:compactDart}.

We consider the payoff function $f$ which is $2/p$ on $N$, 1 on
$[-2,2]^c$ and  0 otherwise.

From distance 1, we can get payoff $2 + \mathbb{P}(|X|>2) $ by aiming at the origin.
Assume now that we are at  distance $t<1$.  If we aim at some point in $ [-1-t, 1+t]^c$, then we cannot
hit the set $N$ noting that any translate and scaling of $\mu_{ac}$ gives probability 0 to $N$.
Therefore, in this case, our expected payoff would be at most 1.
On the other hand, if we aim at some point in $[-1-t, 1+t]$, our payoff would 
be at most $2 + \mathbb{P}(|X|> \frac{1}{t}-1)$. For $t$ sufficiently small, $\frac{1}{t}-1>2$ and
$X$ will have some mass between distance 2 and distance $\frac{1}{t}-1$. For such $t$ we will score worse at distance $t$ than at the larger distance 1, and therefore $(X,f)$ is not reasonable.
\end{proof}

\section{Closure properties}  \label{section:closure} 

In this section, we prove Theorems \ref{theorem:closure}
and \ref{theorem:convergenceInDistribution}.

\subsection{Independent sum of reasonable darts}

\begin{proof} [Proof of Theorem \ref{theorem:closure}]
The theorem in fact is true for any class $\calF$ satisfying the following property.
\begin{equation}
    \left\{h:\mathbb{R}^n\to \mathbb{R}\ |\ h(x):=Ef(x+\sum_{i=1}^m d_iX_i),\ d_i\geq 0 \ \forall i,\ \ f\in \calF{}, \right\}\subseteq \calF{}
    \label{eq:stayInFunctionSet},
\end{equation}
for any independent darts $X_1,....,X_m$. One easily checks that the classes of payoff functions
listed in the statement of the theorem all satisfy this property. We now prove the theorem for
any class $\calF$ satisfying this property.

It is easily seen by induction that it suffices to prove the $m=2$ case. Let $X,Y\in \mathcal{X} _\calF$.
Fix an $f\in \calF{}$, and choose $d_1 , d_2 ,D_1,D_2\geq 0$ such that $d_1\leq D_1, d_2\leq D_2$ and define the function
\begin{equation}
    h(x)=Ef(x+d_1X).
\end{equation}
Since $h\in \calF{}$ by assumption, we have that $Y$ is reasonable with respect to $h$, and thus
\begin{equation}
    \sup_a Eh(a+d_2Y)\geq \sup_a Eh(a+D_2Y)
\end{equation}
and from this we get, using the independence of $X$ and $Y$, that
\begin{equation}
    \begin{aligned}
        \sup_a Ef(a+d_1X+d_2Y)
        &=
        \sup_a \int Ef(a+d_1X+d_2y)\mathrm{d}\mu_Y(y)
        \\&=
        \sup_a \int h(a+d_2y)\mathrm{d}\mu_Y(y)
        = 
        \sup_a Eh(a+d_2Y)
        \\&\geq 
        \sup_a Eh(a+D_2Y)
        =
        \sup_a \int h(a+D_2y)\mathrm{d}\mu_Y(y)
        \\&=
        \sup_a \int Ef(a+d_1X+D_2y)\mathrm{d}\mu_Y(y)
        \\&=
        \sup_a Ef(a+d_1X+D_2Y).
    \end{aligned}
\end{equation}
And by using the same argument again with $X$ and $Y$ reversed, it follows that 
\begin{equation}
    \sup_a Ef(a+d_1X+d_2Y)\geq \sup_a Ef(a+D_1X+D_2Y)
\end{equation}
as desired. This clearly implies that $X+Y\in \mathcal{X} _\calF$.
\end{proof}

\subsection{Convergence in distribution}

\begin{proof} [Proof of Theorem \ref{theorem:convergenceInDistribution}] 
(i) Let $f\in C_0(\mathbb{R}^n)$ and fix $d>0$. We will begin by showing that $\liminf_{j\to\infty} \g{X_j,f}(d) \geq \g{X,f}(d)$ which holds even if $f$ is only bounded. To see this, one
fixes $a\in \mathbb{R}^n$, notes that 
\begin{equation}
    \liminf_{j\to\infty} \g{X_j,f}(d) \geq
\liminf_{j\to\infty}  Ef(a +d X_j)= Ef(a +d X)
\end{equation}
and then takes a supremum over $a$. 

The other direction requires a little more work and uses the fact that $f\in C_0(\mathbb{R}^n)$.
One first  notes (by aiming near infinity) that for any dart $Y$,  $\g{Y,f}(d) \geq 0$.
One also easily notes that if  $\g{Y,f}(d) > 0$, then the supremum in the definition of
$\g{Y,f}(d) $ is obtained. 

Now if $\g{X_j,f}(d)$ does not go to $\g{X,f}(d)$, then there exists $\epsilon_0>0$ and some subsequence $\{j_k\}$
such that 
\begin{equation}
\lim_{k\to\infty}   \g{X_{j_k},f}(d)=\g{X,f}(d)+\epsilon_0.
    \label{eq:subsequence}
\end{equation}

By the above statements,  the supremum in the definition of  $\g{X_{j_k},f}(d)$ 
is obtained for large $k$; let $a_{j_k}$ be some such point.
Using tightness of $\{X_n\}_{n=1}^\infty$, one has that
$$
\lim_{a\to\infty} Ef(a+dX_j)=0
$$
uniformly in $j$ from which it easily follows that $\{a_{j_k}\}$ is contained in a bounded set.
One can therefore extract a further subsequence which converges to some $a_\infty$.
We may assume to simplify the notation that the whole sequence converges.

We now have that $a_{j_k}+X_{j_k}$ converges to $a_\infty+X$  in distribution and hence
$$
\lim_{k\to\infty} \g{X_{j_{k}},f}(d)= 
\lim_{k\to\infty} Ef(a_{j_{k}}+dX_{j_{k_t}})=
Ef(a_\infty+dX)\le \g{X,f}(d)
$$
giving a contradiction.

\medskip
(ii) Let $X_k$ be uniform distribution on $\{0,1/k,2/k,\ldots,1\}$ and $X$ be uniform distribution on
$[0,1]$. Let $f$ be the following bounded continuous function.
It will be zero on $\left(\bigcup_{m=0}^\infty [2m-\frac{1}{10}, 2m+1+\frac{1}{10}]\right)^c$
and on $[2m-\frac{1}{10}, 2m+1+\frac{1}{10}]$ it will be any bounded continuous function
satisfying (1) it takes values in $[0,1]$, (2) it takes the value 1 at each of the points
$\{2m,2m+1/m,2m+ 2/m,\ldots,2m+1\}$, (3) it takes the value 0 at the two endpoints of the interval
and (4) the set of points in the interval where $f$ is zero, has Lebesgue measure at least $.9$.
We will then have that for every $k$, $\g{X_k,f}(1)=1$ (by aiming at $2k$) while it is clear that $\g{X,f}(1)\le 1/2$.

\medskip
(iii) is very easy and left to the reader.

\medskip
(iv)  By (i), it follows that $X$ is reasonable with respect to $C_0(\mathbb{R}^n)$.
It follows now from Proposition \ref{proposition:compactPayoff} that
$X$ is reasonable with respect to $C_b(\mathbb{R}^n)$ as desired.
\end{proof}

\section{Reasonable payoff functions}  \label{section:functions} 

In this section, we prove Proposition \ref{proposition:reasonable.function}

\begin{proof} [Proof of Proposition \ref{proposition:reasonable.function}]
Fix a dart $X$ taking values in $\mathbb{R}^n$. For any $0<d_1<d_2$ we have by the 
weak unimodality of $f$ that
\begin{equation}
    \begin{aligned}
        g_{_{X,f}}(d_1)&=\sup_a Ef(a+d_1X)
        =
        \sup_a Ef\left(\frac{d_1a}{d_2}+d_1X\right)
        \\&=
        \sup_a Ef\left(\frac{d_1}{d_2}(a+d_2X)\right)
        \geq 
        \sup_a Ef(a+d_2X)
        =
        g_{_{X,f}}(d_2).
    \end{aligned}
\end{equation}
\end{proof}

\begin{remark}
If $f:\mathbb{R}^n\to \mathbb{R}$ has the property that there exists $y$ in $\mathbb{R}^n$ 
such that $f(rx+y)$ is decreasing in $r$, where $r$ in $[0,\infty)$, for all $x$ in 
$\mathbb{R}^n$, then $f$ is of the same type as a weakly unimodal function, and is thus reasonable.
\end{remark}

Note also that if $f(x)=\text{arctan}(||x||)$, then it is easy to check that for any dart $X$, 
$g_{_{X,f}}(d)=\sup_x f(x)$ for all $d$ implying that $(X,f)$ is reasonable. The example can be extended by noting 
that any payoff function $f$ with the property that there exist arbitrarily large balls in $\mathbb{R}^n$ 
where $f(x)$ is arbitrarily close to $\sup_{x\in\mathbb{R}^n}f(x)$ would also be reasonable for any dart
$X$ and that $g_{_{X,f}}(d)=\sup_x f(x)$ for all $d$.


\section{Questions}\label{section:open}

\begin{question}
Is there a dart which is reasonable w.r.t.\ all bounded payoff functions, but not all payoff functions?
\end{question}

\begin{question}
Is there a dart which is reasonable w.r.t.\ all (bounded) continuous payoff functions, but not all 
(bounded) payoff functions?
\end{question}

\begin{question}
If $(X,f)$ and $(Y,f)$ are reasonable, does it follow that
$(X+Y,f)$ is reasonable (where of course $X$ and $Y$ are independent)?
\end{question}

\begin{question}
Are there reasonable payoff functions which are not of the same type as 
a weakly unimodal function nor have the behavior exemplified by
$\arctan(x)$?
\end{question}

\begin{question}
Is  the second statement in Theorem \ref{theorem:convergenceInDistribution}(i)  true if the payoff function is only assumed to be continuous and
bounded?
\end{question}

\begin{question}
\label{question:notSelfdecomposable}
Is there an example of a reasonable dart which is not selfdecomposable?
\end{question}

\begin{question}
\label{question:notUnimodal}
Is there an example of a  reasonable one-dimensional dart which is not unimodal?
\end{question}

\begin{question}
\label{question:notInfDiv}
Is there an example of a reasonable dart which is not infinitely divisible?
\end{question}

\begin{question}
\label{question:notAbsCont}
Is there an example of a reasonable nondegenerate one-dimensional dart which is not absolutely continuous?
\end{question}

\noindent
Note that if the answer to Question \ref{question:notUnimodal},  \ref{question:notInfDiv}, or 
\ref{question:notAbsCont} is yes, then this would imply that the answer to 
Question \ref{question:notSelfdecomposable} is also yes.

\section*{Acknowledgements}

We thank Jean Bertoin, Lennart Bondesson, Peter Hallum,  Russ Lyons,
Hjalmar Rosengren and Tamer Tlas
for help with some questions and references.
The second author acknowledges the support of the Swedish Research 
Council, grant no.\ 2016-03835.

\addcontentsline{toc}{section}{References}
\thispagestyle{empty}
\bibliographystyle{plain}

\end{document}